\documentclass[11pt,a4paper]{article}

\usepackage{amsfonts,amsmath,amsthm,amssymb,mathtools}
\usepackage[margin=1.1in]{geometry}
\usepackage{enumitem}
\usepackage{mystyle}
\usepackage{hyperref}

\hypersetup{colorlinks=true,citecolor=blue,linkcolor=blue,urlcolor=blue}

\numberwithin{equation}{section}

\newtheorem{thm}{Theorem}[section]
\newtheorem{prop}[thm]{Proposition}
\newtheorem{cor}[thm]{Corollary}
\newtheorem{lem}[thm]{Lemma}
\newtheorem{rem}[thm]{Remark}
\theoremstyle{definition}
\newtheorem{defn}[thm]{Definition}
\newtheorem{assumption}[thm]{Assumption}

\newcommand{\abs}[1]{\left\lvert #1 \right\rvert}
\newcommand{\norm}[1]{\left\lVert #1 \right\rVert}
\newcommand{\loc}{\mathrm{loc}}
\newcommand{\Dcal}{\mathbb{D}}
\newcommand{\esssup}{\operatorname*{esssup}}
\newcommand{\Uc}{\mathcal{U}}
\newcommand{\Pc}{\mathcal{P}}
\newcommand{\T}{T}
\newcommand{\Lc}{\mathcal{L}}
\newcommand{\Sc}{\mathcal{S}}
\newcommand{\transpose}{\mathsf{T}}

\begin{document}

\title{Densities for scalar-valued BSDEs via unique continuation and backward uniqueness}

\author{
Solesne Bourguin\thanks{LAMA, Universit\'e Gustave Eiffel. Email: \texttt{solesne.bourguin@univ-eiffel.fr}.}
\and
Daniel Schwarz\thanks{Mathematics Department, University College London. Email: \texttt{d.schwarz@ucl.ac.uk}.}
}

\date{}

\maketitle

\begin{abstract}
We give sufficient conditions ensuring that, at every fixed positive time,
the scalar backward component of a Markovian forward-backward stochastic
differential equation with multidimensional forward process admits a density
with respect to Lebesgue measure. Existing density criteria for BSDEs often
obtain Malliavin non-degeneracy through sign or monotonicity assumptions. We
develop a different route for a scalar backward component with an
arbitrary-dimensional forward state. Under regularity assumptions and a
structural compatibility condition on the generator, the terminal condition is
only required to be non-constant. The key idea is to deduce Malliavin
non-degeneracy from deterministic rigidity of the critical set of the
decoupling field. We combine Malliavin calculus with unique continuation and
backward uniqueness for the associated semilinear parabolic equation. Unique
continuation precludes the spatial gradient of the decoupling field from
vanishing on a set of positive measure unless it vanishes identically on that
time slice, while backward uniqueness propagates such vanishing to the terminal
time. Along the way, we establish a unique continuation property from sets of
positive measure and a backward uniqueness result on the whole space for the
linear parabolic systems arising from differentiated semilinear equations.
\end{abstract}

\noindent\textbf{Key words and phrases.}
Backward stochastic differential equations; Malliavin calculus;
absolute continuity; Malliavin non-degeneracy; decoupling fields;
semilinear parabolic equations; unique continuation; backward uniqueness.

\medskip

\noindent\textbf{MSC2020 subject classifications.}
Primary 60H10, 60H07; Secondary 35B60, 35K58, 35K40.

\section{Introduction}
Since the pioneering work of Pardoux and Peng \cite{PP,PP2}, backward stochastic
differential equations have become a fundamental tool in stochastic
analysis. They arise naturally in stochastic control, in mathematical
finance, and in the probabilistic representation of nonlinear partial
differential equations.

\subsection{Problem and setting}
We work in the Markovian framework, where the solution of a
BSDE is linked to a semilinear parabolic equation through a nonlinear
Feynman--Kac formula, and this connection provides a natural bridge
between probabilistic and PDE methods, which plays a fundamental role in this paper. Specifically, we consider the Markovian forward-backward system
\begin{equation}\label{eq:intro_fbsde}
\left\{
\begin{aligned}
X_t &= x_0 + \int_0^t b(s,X_s)\Id s + \int_0^t \sigma(s,X_s)\Id W_s, \\
Y_t &= g(X_T) + \int_t^{T} f(s,X_s,Y_s,Z_s)\Id s - \int_t^{T} Z_s\Id W_s,
\end{aligned}
\right.
\qquad t\in [0,T],
\end{equation}
where $W$ is a $d$-dimensional Brownian motion and $Y$ is
real-valued. We are interested in the question of whether, for a fixed time $t$, the law of $Y_t$ is absolutely continuous with respect to the Lebesgue measure. Absolute continuity of this intermediate-time law is the basic regularity input whenever the distribution of $Y_t$ is used quantitatively -- for threshold events, quantiles, or calibration targets -- and Section~\ref{sec:applications} illustrates the result in two multidimensional models arising in incomplete-market valuation and in cardiac electrophysiology.

\subsection{Main result and proof strategy}
Our main result, Theorem~\ref{thm:main}, may be stated informally as follows: under the regularity and structural assumptions of Section~\ref{sec:framework} -- but with no sign or monotonicity conditions on the data -- if the terminal condition \(g\) is non-constant, then the law of \(Y_t\) is absolutely continuous with respect to the Lebesgue measure for every \(t\in\left(0,T\right)\), in arbitrary dimension of the forward state.

The structural compatibility condition that is specific to our mechanism is, in essence,
\begin{equation*}
\nabla_x f(t,x,y,0)=0,\qquad
\abs{\nabla_x f(t,x,y,z)-\nabla_x f(t,x,y,\zeta)}\le C\abs{z-\zeta}.
\end{equation*}
It is this condition that makes the differentiated equation for \(u=\nabla_xv\) homogeneous at the zero set of \(u\) and hence amenable to the rigidity argument. Remark~\ref{rem:struct_fx_necessary} explains why a restriction of this type cannot simply be dropped from the present mechanism.

For BSDEs, density questions are often approached through Malliavin calculus together with sign conditions on derivatives of the terminal condition and of the generator. See in particular Mastrolia, Possama\"i and R\'eveillac \cite{MPR,MPR2}, Aboura and Bourguin \cite{AB}, and Antonelli and Kohatsu-Higa \cite{AKH}. Density results are also available in multidimensional and coupled settings, as discussed below. Against this background, the distinction of the present work lies in the source of non-degeneracy. Our approach still uses Malliavin calculus to formulate the final non-degeneracy criterion, but proves that criterion through PDE rigidity of the decoupling field, namely unique continuation and backward uniqueness. More precisely, let $\left(v,w\right)$ be a Markovian decoupling field, so that $Y_t=v\left(t,X_t\right)$. Under Assumption~\ref{ass:rigidity}, Proposition~\ref{prop:u_bound} yields a global bound on $\nabla_xv$ on $[0,T)$, and Lemma~\ref{lem:malliavin_rep} below then gives the Malliavin differentiability of $Y_t$ together with
\begin{equation*}
D_rY_t
=
\nabla_x v(t,X_t)D_rX_t\mathbf 1_{[0,t]}(r).
\end{equation*}
The Bouleau--Hirsch criterion then establishes the existence of a density under the condition that
\begin{equation*}
\norm{DY_t}_{L^2([0,T])} > 0
\quad \text{a.s.}
\end{equation*}
Under the ellipticity assumptions on the forward diffusion,
this is equivalent to showing that $\nabla_x v\left(t,X_t\right)\neq 0$ almost
surely for every $t\in \left(0,T\right)$. The problem is thereby converted from a probabilistic one to a deterministic one: how large can the zero set of the spatial gradient of the decoupling field be? The key object is hence the vector-valued function
\begin{equation*}
u := \nabla_x v.
\end{equation*}
After differentiating the semilinear PDE satisfied by $v$, one finds
that $u$ solves a linear parabolic system. The structural assumption on the generator ensures that this system admits a differential inequality of the form
\begin{equation*}
\abs{\Pc u} \leq C\left( \abs{u}+\abs{\nabla_xu} \right),
\end{equation*}
where $\Pc$ is the relevant second-order parabolic operator. This is
the point at which PDE methods enter. The role of Malliavin calculus
is to reduce the density problem to a non-degeneracy condition on
$\nabla_x v$, while the role of the PDE analysis is to show that this gradient cannot vanish on any set of positive Lebesgue measure. The latter requires two rigidity properties: unique continuation and backward uniqueness. The first, unique continuation, is classically formulated under the assumption that the solution of a parabolic equation vanishes on an open set. This does not suffice in our setting: since the law of $X_t$ is absolutely continuous with respect to the Lebesgue measure, the relevant condition is that $\nabla_x v(t, \cdot)$ does not vanish on any set of positive measure. What is needed, therefore, is unique continuation from measurable sets of positive measure, not merely from open sets. This is established in Lemma~\ref{lem:uc}, where we prove that if for some $t_0\in \left(0,T\right)$ the set
\begin{equation*}
\{x\in \R^d \colon \nabla_x v(t_0,x)=0\}
\end{equation*}
has positive Lebesgue measure, then in fact $\nabla_x
v\left(t_0,\cdot\right)\equiv 0$ on $\R^d$. This extension from open sets to
positive-measure sets is one of the PDE contributions of the paper that
may be of independent interest.
\\

The next step is to establish a backward uniqueness result for the gradient $\nabla_x v\left(t_0,\cdot\right)$. Backward uniqueness propagates the vanishing of a function on a time slice of the parabolic cylinder up to the terminal time. The existing PDE literature provides such results on the whole space for scalar
parabolic equations satisfying differential inequalities of the form
$\abs{\Pc \phi}\leq C\left(\abs{\phi}+\abs{\nabla_x\phi}\right)$ under
suitable boundedness assumptions. See, for example, \cite{jWu2019}. However, in our setting, differentiating the semilinear PDE satisfied by
$v$ leads to a linear parabolic system for the vector-valued function $u=\nabla_x v$, so what is needed is a backward uniqueness statement for systems on $\R^d$. This is precisely what Lemma~\ref{lem:bu} provides: it extends the scalar whole-space result to the class of linear parabolic systems arising from the differentiated decoupling field of the BSDE. More concretely, Lemma~\ref{lem:bu} shows that if a vector-valued function $u$ vanishes on a single time slice and satisfies the corresponding differential inequality, then $u$ must vanish on the whole slab $\left[t_0,T\right)\times \R^d$. To apply this result to $u=\nabla_x v$, one must verify that the differentiated equation has the required form and that the boundedness assumptions of Lemma~\ref{lem:bu} are satisfied. These properties are provided by the positive-time regularity estimates established earlier, in particular Proposition~\ref{prop:extra_reg} and Corollary~\ref{cor:slab_bounds}.
\\

With these two rigidity results in hand, one can argue by
contradiction to show that the zero set of $\nabla_x v\left(t,\cdot\right)$
cannot have positive Lebesgue measure. Suppose that for some $t_0\in
\left(0,T\right)$ the set $\{x\in \R^d \colon \nabla_x v\left(t_0,x\right)=0\}$ has
positive Lebesgue measure. Lemma~\ref{lem:uc} then forces
$\nabla_x v\left(t_0,\cdot\right)\equiv 0$ on $\R^d$, and Lemma~\ref{lem:bu} propagates this vanishing to the slab $\left[t_0,T\right)\times \R^d$, so that $v\left(t,\cdot\right)$ is spatially constant for every $t\in \left[t_0,T\right)$. Passing to the terminal time shows that $g$ must be constant, contradicting the non-constancy of \(g\) assumed in Theorem~\ref{thm:main}. This yields Proposition~\ref{prop:mm_inv}, which states that for every $t\in \left(0,T\right)$ the set
\begin{equation*}
\{x\in \R^d : \nabla_x v(t,x)=0\}
\end{equation*}
has Lebesgue measure zero. Since $X_t$ admits a density, it follows
that $\nabla_x v\left(t,X_t\right)\neq 0$ almost surely. Returning to
the Malliavin representation of $DY_t$, this implies
$\norm{DY_t}_{L^2\left(\left[0,T\right]\right)}>0$ almost surely, which yields our main result Theorem~\ref{thm:main}.
\\

In this way, the proof of the density result emerges from a genuine interplay between probabilistic and analytic ideas. Malliavin calculus reduces the problem to the non-vanishing of the gradient of the BSDE decoupling field, while PDE analysis supplies the rigidity needed to rule out pathological degeneracies. The overall strategy is therefore to translate a probabilistic non-degeneracy question into a deterministic one for the associated parabolic equation, and then to carry the resulting information back into the stochastic framework via the Bouleau--Hirsch criterion. The conceptual novelty lies not in a new Malliavin criterion, but in a new way of verifying its non-degeneracy hypothesis through deterministic rigidity of the critical set of the decoupling field.

\subsection{Comparison with the existing literature}

The works \cite{AB,AKH,MPR} and the present paper address the same
broad question -- absolute continuity of marginals of Markovian BSDEs
-- but they do so in substantially different regimes and with fundamentally different mechanisms. In \cite{AB,AKH,MPR}, the forward diffusion is one-dimensional and the non-degeneracy analysis is performed directly at the Malliavin level, by proving the positivity or negativity of $D_rY_t$ or of second-order Malliavin derivatives under sign or convexity assumptions on derivatives of the terminal condition and of the generator. This point of view is extremely effective in dimension one and leads to sharp criteria for the existence of a density at a fixed time, and, in \cite{AB,MPR}, quantitative information on the density itself.
\\

Beyond the one-dimensional Markovian setting, Chertovskih and Shamarova \cite{CS} obtain Gaussian-type density bounds for components of multidimensional BSDEs through a positivity argument, while Likibi Pellat and Menoukeu Pamen \cite{LPMP} analyze density questions for coupled forward--backward systems with irregular forward drifts. These works address genuinely multidimensional density questions, but through mechanisms and hypotheses different from those used here. Against this background, the novelty of the present paper is not multidimensionality by itself. Rather, we introduce a rigidity-based route to Malliavin non-degeneracy for a scalar Markovian BSDE with an arbitrary-dimensional state: pointwise sign or positivity arguments are replaced by a deterministic analysis of the critical set of the decoupling field. Accordingly, the proof relies neither on a one-dimensional order structure nor on monotonicity or sign assumptions on the terminal datum and the generator of the kind that usually appear in existing density criteria for $Y_t$. Under the structural and regularity assumptions of our framework, the terminal datum is only required to be non-constant. To the best of our knowledge, this is the first density criterion in a genuinely multidimensional Markovian setting that derives non-degeneracy from this qualitative non-constancy condition through unique continuation and backward uniqueness.
\\

This multidimensional aspect is particularly important in
applications. In many models of current interest, the state variable records
several factors simultaneously --- multiple assets, stochastic
volatility coordinates, interest-rate factors, inventory or control
variables, hidden factors, or augmented states used to Markovianize
path-dependent claims. In such situations, there is typically no
natural reduction to a single scalar state variable that preserves both the Markov structure and the form of the generator. The
one-dimensional sign method does not extend naturally to this regime,
as it exploits scalar ordering properties of first- and
second-order derivatives that are not available in higher dimensions. By contrast, the PDE rigidity argument is formulated intrinsically in terms of the full vector field $\nabla_x v$ and therefore scales to arbitrary state dimension.
\\

It is equally important to delineate what the one-dimensional literature can do and what our theorem does not aim to cover. On the one hand, the criteria developed in \cite{AB,AKH,MPR} apply to classes of generators with a more general explicit dependence on the state variable than is allowed here. In our approach, this restriction originates from the PDE mechanism underlying the proof: after differentiating the semilinear equation satisfied by the decoupling field, one obtains a linear system for $u=\nabla_x v$ for which the rigidity argument described above is carried out. It is precisely here that our compatibility condition on $\nabla_x f$ enters. On the other hand, the hypotheses in \cite{AB,AKH,MPR} are tailored to one-dimensional sign tests. As a consequence, they can capture situations in which the existence of a density at a fixed time follows from a favourable local sign configuration, even when a global rigidity phenomenon may not hold. By contrast, our theorem is qualitative in nature and rests on a different mechanism: it replaces one-dimensional sign conditions with a multidimensional rigidity argument, at the price of a more restrictive structural assumption on the generator.
\\

Conversely, the scope of our method is not captured by the existing one-dimensional criteria. We impose no sign conditions on derivatives of the terminal datum or of the generator of the kind that typically underlie those results. The terminal condition is only required to be non-constant. Once the structural and regularity assumptions of our framework are satisfied, the existence of a density is established uniformly for every \(t\in \left(0,T\right)\), and in arbitrary state dimension. In particular, the theorem covers genuinely multidimensional Markovian models that are beyond the reach of one-dimensional density criteria, including basket and spread-type payoffs, multifactor stochastic-volatility models, factor-augmented semilinear pricing equations, equilibrium pricing in the spirit of Radner equilibria in incomplete markets \cite{SchwarzRadner}, and lifted Markovian formulations of path-dependent problems.
\\

The present result is therefore best viewed as complementary to the existing sign- and positivity-based approaches. Those methods can yield substantially finer information on the density under structural assumptions tailored to Malliavin positivity. Our contribution is different: a PDE rigidity principle turns the qualitative assumption that \(g\) is non-constant into a fixed-time density result for \(Y_t\), even when the forward state is genuinely multidimensional and no useful pointwise sign condition is available. In this sense, the paper provides a new route to non-degeneracy rather than a direct extension of the one-dimensional theory.

\subsection{Organization of the paper}

The paper is organized as follows. Section~\ref{sec:framework}
introduces the Markovian setting, states the assumptions, and
formulates the main theorem. Section~\ref{sec:regularity} establishes
the basic regularity of the decoupling field and then derives the
stronger coefficient-based regularity statements needed for the later
PDE analysis. Section~\ref{sec:rigidity} establishes the unique
continuation and backward uniqueness results and culminates in the
key rigidity proposition for the zero set of $\nabla_x
v$. Section~\ref{sec:main_proof} is dedicated to the probabilistic proof of the density theorem. Finally, Section~\ref{sec:applications} illustrates the scope of the result through two multidimensional examples, arising respectively from multifactor incomplete-market valuation with stochastic volatility, interest-rate, and credit factors, and from a viscous-eikonal model for cardiac activation mapping.

\section{Framework, assumptions, and main theorem}\label{sec:framework}
In this section, we introduce the Markovian forward-backward system studied throughout the paper, state the standing assumptions, and formulate the main density result. Since our proof combines stochastic arguments with PDE techniques, it is useful from the outset to separate the assumptions according to the role they play. Some hypotheses are needed at the level of the forward diffusion and of the Markovian representation of the BSDE, while others enter only later, when we study the differentiated equation for the decoupling field and invoke the PDE rigidity tools developed in the paper.

\subsection{The Markovian framework}
We work on a finite time horizon $\left[0,T\right]$, with $T>0$ fixed. For each
initial condition $\left(s,x\right)\in \left[0,T\right]\times \R^d$, we consider a forward
diffusion $X^{s,x}$ started from $x$ at time $s$, together with a
backward component $\left(Y^{s,x},Z^{s,x}\right)$ driven by the same Brownian
motion. The forward equation determines the state process, while the
BSDE encodes the quantity of interest whose marginal law at a fixed
time we want to study. More precisely, for $\left(s,x\right)\in \left[0,T\right]\times \R^d$, let $X^{s,x}$ be the solution of
\begin{equation}\label{eq:SDE}
X_t^{s,x} = x + \int_s^t b(r,X_r^{s,x})\Id r + \int_s^t \sigma(r,X_r^{s,x})\Id W_r,
\qquad t\in [s,T],
\end{equation}
and let $\left(Y^{s,x},Z^{s,x}\right)$ solve
\begin{equation}\label{eq:BSDE}
Y_t^{s,x} = g(X_T^{s,x}) + \int_t^{T} f(r,X_r^{s,x},Y_r^{s,x},Z_r^{s,x})\Id r - \int_t^{T} Z_r^{s,x}\Id W_r,
\qquad t\in [s,T].
\end{equation}
The Markovian setting means that the backward variables can be represented as deterministic functions of the current time and current state. This is the natural framework for our analysis, since it allows one to pass from the BSDE to the associated semilinear parabolic equation and to work with the corresponding decoupling field.

\begin{defn}\label{defn:markov_sol}
 A \emph{Markovian solution} of \eqref{eq:BSDE} is a pair of measurable functions
\begin{equation*}
v \colon [0,T]\times \R^d \to \R,
\qquad
w \colon [0,T]\times \R^d \to \R^d,
\end{equation*}
such that for every $\left(s,x\right)\in \left[0,T\right]\times \R^d$,
$Y_t^{s,x}=v\left(t,X_t^{s,x}\right)$ for all $t\in[s,T]$, while
\begin{equation*}
Z_t^{s,x}=w\left(t,X_t^{s,x}\right)
\qquad \text{for }\Id t\otimes\Id\mathbb P\text{-almost every }(t,\omega)\in[s,T]\times\Omega.
\end{equation*}
With this convention, \eqref{eq:BSDE} holds with
\begin{equation*}
\E\left(\int_s^T \abs{f(r,X_r^{s,x},Y_r^{s,x},Z_r^{s,x})}\Id r + \int_s^T \abs{Z_r^{s,x}}^2\Id r\right) < \infty.
\end{equation*}
\end{defn}
Throughout the paper we fix an initial point $x_0\in\R^d$ and, when no initial condition is displayed, write $X_t:=X_t^{0,x_0}$, $Y_t:=Y_t^{0,x_0}$, and $Z_t:=Z_t^{0,x_0}$. The density question addressed in this paper concerns the law of $Y_t=v\left(t,X_t\right)$ at fixed positive times. In particular, the entire argument will revolve around the regularity and non-degeneracy properties of the decoupling field $v$ and, more specifically, of its spatial gradient.

\subsection{Assumptions}
We now list the assumptions used in the sequel. Rather than presenting them as a single block, we separate them according to the stage of the argument in which they enter. This is important for the logic of the paper: the stochastic-flow analysis, the existence of a sufficiently regular Markovian solution, the positive-time regularity theory, and the PDE rigidity step do not all require the same hypotheses. The assumptions are nested: each includes those named in its preamble,
so that hypothesis lists in the sequel cite only the strongest
assumptions required. In particular, Assumption~\ref{ass:rigidity}
entails all of the others.

Throughout the paper we fix a H\"older exponent $\alpha\in\left(0,1\right)$.
All parabolic H\"older conditions below are understood with
respect to this common exponent, namely with exponent $\alpha/2$ in the time
variable and $\alpha$ in the spatial variables.

We begin with assumptions on the forward coefficients. These are the natural hypotheses ensuring that the forward diffusion is well posed and sufficiently regular with respect to the initial condition. They are the basic input for the stochastic side of the argument.

\begin{assumption}[Forward flow coefficients]\label{ass:flow}
The coefficients $b \colon \left[0,T\right]\times \R^d \to \R^d$ and
$\sigma \colon \left[0,T\right]\times \R^d \to \R^{d\times d}$ are
continuous in $t$ and continuously differentiable in $x$, with
$\nabla_x b$ and $\nabla_x \sigma$ bounded on
$\left[0,T\right]\times \R^d$.
\end{assumption}

\begin{rem}\label{rem:flow_consequences}
Assumption~\ref{ass:flow} already implies that $b$ and $\sigma$ are
globally Lipschitz in $x$, uniformly with respect to $t$, and of at
most linear growth: the Lipschitz bound follows from the boundedness
of the spatial gradients via the integral form of the mean value
theorem, and the growth bound then follows because
$t\mapsto \bigl(b(t,0),\sigma(t,0)\bigr)$ is continuous on the compact
interval $\left[0,T\right]$, hence bounded.
\end{rem}

\begin{lem}\label{lem:jacobian_flow}
Under Assumption~\ref{ass:flow}, for every $0\le s\le t\le T$ the map
$x\mapsto X^{s,x}_t$ is differentiable, and its Jacobian
$J_{s,t}:=\nabla_x X^{s,x}_t$ is a continuous, a.s.\ invertible
$\R^{d\times d}$-valued process solving
\begin{equation}\label{eq:jacobian_sde}
\Id J_{s,t}
= \nabla_x b(t,X^{s,x}_t)\,J_{s,t}\,\Id t
+ \sum_{\ell=1}^d \nabla_x \sigma_\ell(t,X^{s,x}_t)\,J_{s,t}\,\Id W_t^\ell,
\qquad J_{s,s}=I_d,
\end{equation}
where $\sigma_\ell$ is the $\ell$-th column of $\sigma$ and $I_d$ the identity
matrix. Moreover,
\begin{equation}\label{eq:flow_moments}
\E\left( \sup_{t\in [s,T]}\bigl( \norm{J_{s,t}}^q + \norm{J_{s,t}^{-1}}^q \bigr) \right)
< \infty
\qquad \text{for every } q\ge 1.
\end{equation}
\end{lem}
 
\begin{proof}
By Assumption \ref{ass:flow} the coefficients $b$ and $\sigma$ are continuously
differentiable in $x$ with bounded derivatives, and by Remark \ref{rem:flow_consequences} of at most
linear growth. The hypotheses of \cite[\S2.8, Theorem~6]{Krylov} therefore
hold with $i=1$, which gives the differentiability of $x\mapsto X^{s,x}_t$
together with $\E\left(\sup_{t\in [s,T]}\norm{J_{s,t}}^q\right)<\infty$ for every
$q\ge 1$. Differentiating \eqref{eq:SDE} in $x$ and interchanging the
derivative with the integrals yields \eqref{eq:jacobian_sde}.
 
To prove invertibility, let $K_{s,t}$ be the unique
  strong solution of the matrix-valued linear SDE
  \begin{equation}\label{eq:inverse_jacobian_sde}
\left\{
\begin{aligned}
\Id K_{s,t}
&= -K_{s,t}\,\nabla_x b(t,X^{s,x}_t)\,\Id t \\
&\quad - \sum_{\ell=1}^d K_{s,t}\,\nabla_x \sigma_\ell(t,X^{s,x}_t)\,\Id W_t^\ell \\
&\quad + \sum_{\ell=1}^d K_{s,t}
\bigl(\nabla_x\sigma_\ell(t,X^{s,x}_t)\bigr)^2\,\Id t, \\
K_{s,s}&=I_d.
\end{aligned}
\right.
\end{equation}
The coefficients in \eqref{eq:inverse_jacobian_sde} are bounded by
Assumption~\ref{ass:flow}. Standard Burkholder--Davis--Gundy and
Gronwall estimates for linear SDEs give
\begin{equation}\label{eq:K_moments}
\E\left(\sup_{t\in[s,T]}\norm{K_{s,t}}^q\right)<\infty
\qquad\text{for every }q\ge1.
\end{equation}
Applying It\^o's product rule to $K_{s,t}J_{s,t}$ and using
\eqref{eq:jacobian_sde} and \eqref{eq:inverse_jacobian_sde}, the drift
terms and the stochastic terms cancel identically, including the quadratic
covariation term, and therefore
\begin{equation*}
\Id\bigl(K_{s,t}J_{s,t}\bigr)=0,
\qquad
K_{s,s}J_{s,s}=I_d.
\end{equation*}
Thus $K_{s,t}J_{s,t}=I_d$ for every $t\in[s,T]$, almost surely. Since
$J_{s,t}$ and $K_{s,t}$ are square matrices, the existence of this left
inverse implies that $J_{s,t}$ is invertible and
$K_{s,t}=J_{s,t}^{-1}$. Combining this identity with
\eqref{eq:K_moments} gives the inverse-moment bound in
\eqref{eq:flow_moments}.
\end{proof}

The next set of assumptions strengthens the hypotheses on the forward
coefficients in a direction dictated by the PDE part of the
proof. These hypotheses are not needed merely to define the forward
process. Rather, they are the ones that allow us to use classical
parabolic regularity, to obtain absolute continuity of the law of the
forward diffusion at positive times, and later to apply the unique
continuation and backward uniqueness arguments to the differentiated
equation -- see Remark \ref{rem:forward_roles} below.

\begin{assumption}[Forward parabolic coefficients]\label{ass:parabolic}
In addition to Assumption \ref{ass:flow}, assume that the following hold.
\begin{enumerate}
\renewcommand{\labelenumi}{\textup{(P\arabic{enumi})}}
\renewcommand{\theenumi}{P\arabic{enumi}}
\item\label{it:P1} The coefficients $b$ and $\sigma$ are bounded on
  $\left[0,T\right]\times \R^d$.
\item\label{it:P2} The coefficients are regular in time, uniformly in
  space: there exists $C>0$ such that, for all $s,t\in\left[0,T\right]$
  and all $x\in\R^d$,
  \begin{equation*}
  \abs{b(t,x)-b(s,x)} \le C\,\abs{t-s}^{\alpha/2}
  \qquad\text{and}\qquad
  \norm{\sigma(t,x)-\sigma(s,x)} \le C\,\abs{t-s}.
  \end{equation*}
\item\label{it:P3} The matrix $\sigma$ is uniformly elliptic: there exists
  $\lambda>0$ such that
  \begin{equation*}
  \xi^\transpose (\sigma\sigma^\transpose)(t,x)\xi \geq \lambda \abs{\xi}^2,
  \qquad \xi\in \R^d,
  \quad (t,x)\in [0,T]\times \R^d.
  \end{equation*}
\end{enumerate}
\end{assumption}

\begin{rem}\label{rem:forward_roles}
The distinction between Assumptions \ref{ass:flow} and
\ref{ass:parabolic} will be used repeatedly below. Assumption
\ref{ass:flow} is the stochastic-flow hypothesis: it yields
well-posedness of the forward SDE, moment bounds, differentiability of
$X^{s,x}$ with respect to the initial condition, and Malliavin
differentiability of the forward process. Assumption
\ref{ass:parabolic}, by contrast, is tailored to the analytic part of
the proof: it provides classical local regularity for the decoupling
field, Aronson-type absolute continuity of the law of $X_t$ at positive
times, and places the differentiated equation within the scope of the
unique continuation and backward uniqueness results of
Section~\ref{sec:rigidity}.

For later reference, we record the form in which Assumption
\ref{ass:parabolic} enters these arguments. Define the diffusion matrix
\begin{equation}\label{eq:A_def}
A := \tfrac12\,\sigma\sigma^{\transpose},
\end{equation}
a notation in force throughout the paper. Under Assumption
\ref{ass:parabolic}, the matrix $A$ is bounded, with the two-sided
ellipticity bounds
\begin{equation*}
\tfrac{\lambda}{2}\,\abs{\xi}^{2}
\;\le\;
\xi^{\transpose}A(t,x)\,\xi
\;\le\;
\tfrac12\,\norm{\sigma}_{\infty}^{2}\,\abs{\xi}^{2},
\qquad
\xi\in\R^{d},\quad (t,x)\in\left[0,T\right]\times\R^{d}.
\end{equation*}
The gradient $\nabla_{x}A$ is bounded, by \ref{it:P1} and the
boundedness of $\nabla_{x}\sigma$ in Assumption~\ref{ass:flow}, so
that $A$ is globally Lipschitz in $x$, uniformly in $t$. Combined
with the Lipschitz bound in $t$, uniformly in $x$, supplied by
\ref{it:P1} and \ref{it:P2}, it follows that $A$ is globally
Lipschitz in $(t,x)$ on $\left[0,T\right]\times\R^{d}$. In particular, $A$ is uniformly continuous and belongs to
$C^{\alpha/2,\alpha}$ on bounded cylinders. Moreover, since
$t\mapsto A(t,x)$ is Lipschitz uniformly in $x$, the derivative
$\partial_{t}A(\cdot,x)$ exists for almost every $t\in\left[0,T\right]$,
with $\abs{\partial_{t}A}$ bounded almost everywhere by the Lipschitz
constant. This is the form in which the time regularity of $A$ is used
in the backward uniqueness argument of Section~\ref{sec:rigidity},
where $\partial_{t}A$ enters only through integral identities and
almost-everywhere bounds.
\end{rem}

We next turn to the backward equation itself. In the following we collect the assumptions on the terminal datum, on the generator, and on the Markovian solution. In particular, the structural condition on $\nabla_x f$ is the one that will later allow the differentiated decoupling field to satisfy a linear parabolic system amenable to the rigidity argument.

\begin{assumption}[Backward coefficients and Markovian solution]\label{ass:backward}
The BSDE \eqref{eq:BSDE} admits a Markovian solution. We fix one such pair
$(v,w)$ for the remainder of the paper and define
\begin{equation*}
C_Y := \sup_{[0,T]\times\R^d}\abs{v}\in[0,+\infty],
\qquad
C_Z := \sup_{[0,T]\times\R^d}\abs{w}\in[0,+\infty],
\end{equation*}
together with the associated solution strip
\begin{equation*}
\Sc := [0,T]\times\R^d\times\{y\in\R:\abs{y}\le C_Y\}
       \times\{z\in\R^d:\abs{z}\le C_Z\},
\end{equation*}
with the convention that the restriction in $y$ (resp.\ $z$) is dropped when
$C_Y=+\infty$ (resp.\ $C_Z=+\infty$). Assume, moreover, that there exist a
constant $K_f>0$ and, for every $R>0$, a constant $K_{f,R}>0$, such that the
following hold.
\begin{enumerate}
\renewcommand{\labelenumi}{\textup{(B\arabic{enumi})}}
\renewcommand{\theenumi}{B\arabic{enumi}}
\item\label{it:B1} The terminal condition $g\colon\R^d\to\R$ is continuously
differentiable and of at most polynomial growth.
\item\label{it:B2} The generator
$f\colon[0,T]\times\R^d\times\R\times\R^d\to\R$ is continuously
differentiable in $(x,y,z)$. On $\Sc$, it is H\"older continuous in $t$ and
$x$, locally uniformly in $(y,z)$: for all $t,s\in[0,T]$, $x,\xi\in\R^d$, and
all $(y,z)$ with $(t,x,y,z)\in\Sc$ and $\abs{y}+\abs{z}\le R$,
\begin{equation}
\abs{f(t,x,y,z)-f(s,\xi,y,z)}
\le K_{f,R}\left(\abs{t-s}^{\alpha/2}+\abs{x-\xi}^{\alpha}\right).
\label{eq:holder_f_tx}
\end{equation}
Moreover, on $\Sc$,
\begin{equation}
\abs{\partial_y f(t,x,y,z)}+\abs{\nabla_z f(t,x,y,z)}\le K_f.
\label{eq:lip_fyz}
\end{equation}
\item\label{it:B3} The map $\nabla_x f$ satisfies, on $\Sc$,
\begin{align}
\nabla_x f(t,x,y,0)&=0,\label{eq:struct_fx_zero}\\
\abs{\nabla_x f(t,x,y,z)-\nabla_x f(t,x,y,\zeta)}
&\le K_f\abs{z-\zeta}.\label{eq:struct_fx_lip}
\end{align}
\end{enumerate}
\end{assumption}

\begin{rem}\label{rem:CYCZ}
The quantities \(C_Y\) and \(C_Z\) are determined by the Markovian solution
itself and are not fixed in advance. They are the possibly infinite supremum
norms of the decoupling fields \(v\) and \(w\). By
Definition~\ref{defn:markov_sol},
\begin{equation*}
\abs{Y_t^{s,x}}\le C_Y
\qquad\text{for every }t\in[s,T],\quad\text{a.s.},
\end{equation*}
while
\begin{equation*}
\abs{Z_t^{s,x}}\le C_Z
\qquad\text{for }\Id t\otimes\Id\mathbb P\text{-almost every }(t,\omega)\in[s,T]\times\Omega.
\end{equation*}
Thus \(C_Y\) and \(C_Z\) localize the assumptions on the generator to the
region of \((y,z)\)-space visited by the solution, with the natural
\(\Id t\otimes\Id\mathbb P\)-almost-everywhere interpretation for the
\(Z\)-component. When both constants are finite, \ref{it:B2} and
\ref{it:B3} require control of the coefficients only on the strip \(\Sc\).
This localization is used through the truncation maps \(\rho_Y\) and
\(\rho_Z\) introduced in Step~1 of the proof of
Proposition~\ref{prop:reg}. If \(C_Y\) or \(C_Z\) is infinite, the
corresponding restriction disappears and the relevant bounds become global in
that variable.

The existence of a Markovian solution is itself a hypothesis, but sufficient
conditions are classical: if the conditions of \ref{it:B2} and \ref{it:B3}
hold with the strip \(\Sc\) replaced by all of
\([0,T]\times\R^d\times\R\times\R^d\), and \(g\) is continuous and of at most
polynomial growth, then the generator is globally Lipschitz in \((y,z)\) and
of linear growth, and the nonlinear Feynman--Kac construction provides a pair
\((v,w)\) as in Definition~\ref{defn:markov_sol} -- see \cite{PP2} and
\cite[Section~4]{KPQ}.
\end{rem}

The next two sets of assumptions concern the PDE analysis of the differentiated decoupling field. They are separated because they play different roles. Assumption \ref{ass:extra_reg} is used to obtain additional regularity for $\nabla_x v$ and for the coefficients of the linearized equation. Assumption \ref{ass:rigidity} is stronger and is the point at which we impose the uniform bounds needed to implement the rigidity argument at positive times.
\\

We begin with the local regularity assumptions. These hypotheses are local
in nature and enter the argument at exactly one point: they justify
differentiating the semilinear equation for the decoupling field, by
ensuring that the coefficients and the source of the resulting linear
system are locally H\"older continuous, as required by the interior
Schauder theory (see the proof of Proposition~\ref{prop:extra_reg}).

\begin{assumption}[Local regularity of the linearized coefficients]\label{ass:extra_reg}
In addition to Assumptions \ref{ass:flow} and \ref{ass:backward}, assume
that the following hold.
\begin{enumerate}
\renewcommand{\labelenumi}{\textup{(R\arabic{enumi})}}
\renewcommand{\theenumi}{R\arabic{enumi}}
\item\label{it:R1} The maps $\nabla_x b$ and $\nabla_x \sigma$ belong to
$C^{\alpha/2,\alpha}_{\loc}\left(\left[0,T\right]\times \R^d\right)$.
\item\label{it:R2} For every compact set $K\subset \Sc$ there exists a
constant $C_K>0$ such that, for all
$\left(t,x,y,z\right),\left(s,\xi,\eta,\zeta\right)\in K$,
\begin{equation}\label{eq:holder_R2}
\begin{aligned}
&\abs{\nabla_x f(t,x,y,z)-\nabla_x f(s,\xi,\eta,\zeta)}
+\abs{\partial_y f(t,x,y,z)-\partial_y f(s,\xi,\eta,\zeta)}
\\
&\qquad
+\abs{\nabla_z f(t,x,y,z)-\nabla_z f(s,\xi,\eta,\zeta)}
\\
&\qquad\le C_K\left(
\abs{t-s}^{\alpha/2}+\abs{x-\xi}^{\alpha}
+\abs{y-\eta}^{\alpha}+\abs{z-\zeta}^{\alpha}
\right).
\end{aligned}
\end{equation}
\end{enumerate}
\end{assumption}

The final set of assumptions supplies the boundedness input for the
backward uniqueness argument of Section~\ref{sec:rigidity}. By
Definition~\ref{defn:markov_sol} and Remark~\ref{rem:CYCZ}, the Markovian
solution already remains in the strip \(\Sc\), pointwise in time for the
\(Y\)-component and \(\Id t\otimes\Id\mathbb P\)-almost everywhere for
the \(Z\)-component. Assumption~\ref{ass:backward} therefore controls the
generator along the solution. The stronger assumptions below extend the
relevant bounds to the whole space. The boundedness of \(\nabla_x g\) in
\ref{it:U1} makes the global gradient estimate finite, while
\ref{it:U2} and \ref{it:U3} strengthen \eqref{eq:lip_fyz} and
\eqref{eq:struct_fx_zero}--\eqref{eq:struct_fx_lip}. These global bounds
allow us to differentiate the BSDE with respect to the initial state by a
direct difference-quotient argument and to obtain the \(L^q\) estimates
used in Propositions~\ref{prop:u_bound} and \ref{prop:grad_slab_bound}.
They also avoid the need to invoke a parameter-differentiability theorem
with stronger uniformity assumptions in the spatial variable.

\begin{assumption}[Global bounds for the linearized equation]\label{ass:rigidity}
In addition to Assumptions \ref{ass:parabolic}, \ref{ass:backward}, and
\ref{ass:extra_reg}, assume that there exists a constant \(C_f>0\) such
that the following hold.
\begin{enumerate}
\renewcommand{\labelenumi}{\textup{(U\arabic{enumi})}}
\renewcommand{\theenumi}{U\arabic{enumi}}
\item\label{it:U1} The terminal gradient is bounded:
\(\nabla_x g\in C_b\left(\R^d;\R^d\right)\).
\item\label{it:U2} The bound \eqref{eq:lip_fyz} holds on all of
\([0,T]\times\R^d\times\R\times\R^d\), with \(C_f\) in place of \(K_f\):
\begin{equation*}
\abs{\partial_y f(t,x,y,z)}+\abs{\nabla_z f(t,x,y,z)}\le C_f,
\qquad
(t,x,y,z)\in[0,T]\times\R^d\times\R\times\R^d.
\end{equation*}
Moreover, \(\partial_y f\) and \(\nabla_z f\) are Lipschitz continuous
in \((y,z)\), uniformly in \((t,x)\): for all
\((t,x)\in[0,T]\times\R^d\), \(y,\eta\in\R\), and \(z,\zeta\in\R^d\),
\begin{multline*}
\abs{\partial_y f(t,x,y,z)-\partial_y f(t,x,\eta,\zeta)}
+\abs{\nabla_z f(t,x,y,z)-\nabla_z f(t,x,\eta,\zeta)}
\\
\le C_f\left(\abs{y-\eta}+\abs{z-\zeta}\right).
\end{multline*}
\item\label{it:U3} The structural conditions
\eqref{eq:struct_fx_zero}--\eqref{eq:struct_fx_lip} hold on all of
\([0,T]\times\R^d\times\R\times\R^d\), with \(C_f\) in place of \(K_f\):
for all \((t,x,y)\in[0,T]\times\R^d\times\R\) and \(z,\zeta\in\R^d\),
\begin{equation*}
\nabla_x f(t,x,y,0)=0,
\qquad
\abs{\nabla_x f(t,x,y,z)-\nabla_x f(t,x,y,\zeta)}
\le C_f\abs{z-\zeta}.
\end{equation*}
\end{enumerate}
\end{assumption}

\begin{rem}[On the \(x\)-dependence of the generator]\label{rem:struct_fx_necessary}
The structural condition imposed on \(\nabla_x f\) is dictated by the PDE mechanism underlying the proof. The decoupling field \(v\) solves a semilinear parabolic equation of the form
\begin{equation*}
\partial_t v + \Lc v + f\left(t,x,v,(\nabla_x v)\sigma\right)=0,
\qquad
\Lc:=\sum_{i,j=1}^{d}A_{ij}\,\partial_{ij}+\sum_{i=1}^{d}b_{i}\,\partial_{i},
\end{equation*}
where \(\Lc\) is the generator of the forward diffusion, cf.~\eqref{eq:A_def}. Differentiating in space and setting $u=\nabla_x v$ produces a linear parabolic system for \(u\) with the external forcing term
\begin{equation*}
\nabla_x f\left(t,x,v,\left(\nabla_x v\right)\sigma\right).
\end{equation*}
For the unique continuation and backward uniqueness argument to apply, this forcing must vanish whenever \(u\) does. This is precisely the role of the condition
\begin{equation*}
\nabla_x f(t,x,y,0)=0.
\end{equation*}
Combined with the Lipschitz control of \(\nabla_x f\) in the \(z\)-variable, it yields
\begin{equation*}
\abs{\nabla_x f\left(t,x,v,(\nabla_x v)\sigma\right)}\leq C\abs{u},
\end{equation*}
and thereby reduces the differentiated equation to a differential inequality of the form
\begin{equation*}
\abs{\Pc u}\leq C\left(\abs{u}+\abs{\nabla_xu}\right),
\end{equation*}
which is exactly the homogeneous structure required by the unique continuation and backward uniqueness results proved later.
\\

The assumption should therefore be understood as a compatibility condition between the explicit \(x\)-dependence of the generator and the geometry of the critical set of the decoupling field. If the forcing term does not vanish on the set \(\{u=0\}\), the zero set of \(u=\nabla_x v\) is no longer rigid: the inhomogeneous \(x\)-dependence of the generator can cause \(u\) to vanish at intermediate times independently of the terminal condition. From the probabilistic viewpoint, this is precisely the obstruction, since the Malliavin derivative of \(Y_t\) factorizes through \(\nabla_x v\left(t,X_t\right)\), and the argument relies on showing that this quantity cannot vanish on a set of positive measure unless it vanishes identically.
\\

That said, it should not be interpreted as a necessary condition for the existence of a density in general. Rather, it is necessary for any approach that relies on rigidity of the zero set of the gradient of the decoupling field. Other approaches, notably the one-dimensional sign criteria of \cite{AB,AKH,MPR}, can accommodate generators with more general explicit state dependence, though those methods are themselves restricted to scalar state processes and rely on order-theoretic tools that are unavailable in higher dimensions.
\\

The relevance of the structural condition can already be seen on a simple example, which shows that without such a compatibility requirement the conclusion of our theorem may fail altogether. Consider the one-dimensional case \(d=1\), with \(b\equiv 0\) and \(\sigma\equiv 1\). Let \(\phi\) be a smooth function such that $\phi\left(T\right)=1$ and $\phi\left(t_0\right)=0$ for some \(t_0\in \left(0,T\right)\), and set $g\left(x\right)=\sin x$ and
\begin{equation*}
f(t,x,y,z)=\left(\tfrac12 \phi(t)-\phi'(t)\right)\sin x.
\end{equation*}
Then $v\left(t,x\right)=\phi\left(t\right)\sin x$ solves the semilinear PDE associated with the BSDE, with terminal condition \(g\). In particular, \(g\) is clearly non-constant, so the terminal datum carries genuine spatial variation. However,
\begin{equation*}
u(t,x)=\nabla_x v(t,x)=\phi(t)\cos x,
\end{equation*}
and therefore $u\left(t_0,\cdot\right)\equiv 0$. Although the terminal condition is non-constant, the gradient \(u(t_0,\cdot)\) vanishes identically, so that \(v(t_0,\cdot)\) is constant on \(\R\). Consequently, \(Y_{t_0}\) is deterministic and cannot admit a density.
\\

This is exactly the phenomenon that our structural condition is designed to exclude. In the differentiated equation for \(u=\nabla_x v\), the explicit \(x\)-dependence of the generator produces a forcing term that does not vanish when \(u=0\). As a result, the zero set of \(u\) is no longer rigid: even if \(u\) vanishes identically at some time, this does not force the terminal datum to be constant. The example therefore shows that the issue is not merely that the proof would fail without the condition \(\nabla_x f\left(t,x,y,0\right)=0\), but that without some restriction of this type the fixed-time density statement itself is false in general.
\end{rem}

\subsection{Main theorem and remarks}

We are now in a position to state the main result of the paper. The
purpose of the preceding assumptions is to place the Markovian BSDE in
a framework where both sides of the argument can be carried out: on
the one hand, the Malliavin differentiability and Markovian
representation needed to formulate the non-degeneracy question for
\(Y_t\), and on the other hand, the regularity and rigidity properties
of the decoupling field required by the PDE part of the proof. The
theorem below is the final outcome of this interplay. Its proof proceeds in two main stages. First, in Section~\ref{sec:regularity}, we establish the regularity properties of the decoupling field and of its spatial gradient that are needed to differentiate the associated semilinear equation and to control the resulting coefficients. Second, in Section~\ref{sec:rigidity}, we prove the rigidity statements for the differentiated equation, namely the unique continuation and backward uniqueness results that ultimately prevent the spatial gradient of the decoupling field from vanishing on a set of positive measure. Once these ingredients are available, the density statement follows in Section~\ref{sec:main_proof} by combining them with the Bouleau--Hirsch criterion.

\begin{thm}\label{thm:main}
Assume \ref{ass:rigidity}, and suppose that the terminal condition \(g\) is
non-constant. Then, for every \(t\in(0,T)\), the law of \(Y_t^{0,x_0}\) is
absolutely continuous with respect to the Lebesgue measure on \(\R\).
\end{thm}

The theorem is formulated for strictly positive times, and it is useful to briefly explain why the endpoints \(t=0\) and \(t=T\) fall outside its scope. The reason is that these two times are of a different nature. At the initial time, the backward variable is deterministic when the initial condition is deterministic, so no density can exist. At the terminal time, by contrast, the question is no longer governed by the rigidity mechanism developed in this paper, but instead by the regularity properties of the image measure of \(X_T\) under the terminal map \(g\).

\begin{rem}
The terminal time \(t=T\) must be treated separately, since in that case
\begin{equation*}
Y_T = g(X_T),
\end{equation*}
and the density question is governed directly by the regularity of the
pushforward of the law of \(X_T\) through \(g\). Under
Assumption~\ref{ass:parabolic}, the law of \(X_T\) admits a density
\(p_T\left(x\right)\) with respect to the Lebesgue measure, and Aronson's bounds
imply that \(p_T\left(x\right)>0\) for every \(x\in\R^d\) \cite{Aronson}. Thus,
from a measure-theoretic viewpoint, the natural condition for
\(g\left(X_T\right)\) to admit a density is that the pushforward measure
\(g_\#\left(p_T\left(x\right)dx\right)\) be absolutely continuous with respect to the Lebesgue measure on \(\R\). Equivalently, denoting by \(\abs{\cdot}\) the Lebesgue measure (on \(\R^d\) or on \(\R\), according to context), one must have
\begin{equation*}
\abs{ g^{-1}(N) }=0
\quad
\text{for every Borel set } N\subset \R \text{ such that } \abs{N}=0.
\end{equation*}
A convenient sufficient condition is that \(g\) be locally Lipschitz and that $\abs{\nabla g}>0$ almost everywhere. In that case, the coarea formula yields
\begin{equation*}
\mathrm{Law}\left( g(X_T) \right)(dy)=\rho_g(y)dy,
\qquad
\rho_g(y)=\int_{g^{-1}(y)} \frac{p_T(x)}{\abs{\nabla g(x)}}d\mathcal H^{d-1}(x),
\end{equation*}
for almost every \(y\in\R\), where \(\mathcal H^{d-1}\) is the natural surface measure on the level set \(g^{-1}\left(y\right)\) -- see, for instance, \cite[Section~3.4.3]{lEvans2015}. In particular, the bounded-gradient assumption in Assumption~\ref{ass:rigidity}\ref{it:U1} should not be viewed as an optimal condition for the terminal law \(g\left(X_T\right)\). Its role in the present paper is different: it is a convenient coefficient-level hypothesis ensuring the global boundedness of \(u=\nabla_x v\), which is needed in the rigidity argument at strictly positive times.
\end{rem}

\section{Regularity of the decoupling field}\label{sec:regularity}

This section establishes the analytic regularity properties of the
decoupling field that connect the Markovian BSDE with the PDE arguments used
later in the proof. We first show that the backward component is represented
by a sufficiently regular function \(v\) solving the associated semilinear
parabolic equation in the classical sense. This makes it possible to
differentiate the equation and study \(u=\nabla_xv\) through a linear
parabolic system. Under Assumption~\ref{ass:rigidity}, we then obtain the
global bounds needed for the rigidity argument. Once these bounds are
available, the Malliavin chain rule gives the factorization of \(D_rY_t\)
through \(\nabla_xv(t,X_t)\), as shown in
Subsection~\ref{sec:malliavin_rep}.

\subsection{Basic regularity of the decoupling field and identification of the martingale integrand}

We begin with the basic regularity statement that connects the
Markovian BSDE to its associated semilinear PDE and identifies its
martingale integrand in terms of the spatial gradient of the
decoupling field.

\begin{prop}[Basic regularity and identification of the martingale integrand]\label{prop:reg}
Assume \ref{ass:parabolic} and \ref{ass:backward}.
Let \(\left(v,w\right)\) be the Markovian solution fixed in Assumption~\ref{ass:backward}, and recall from \eqref{eq:A_def} that
$A\left(t,x\right) = \frac12\sigma\left(t,x\right)\sigma\left(t,x\right)^\transpose$. Then, the following hold.
\begin{enumerate}
\renewcommand{\labelenumi}{\textup{(\alph{enumi})}}
\renewcommand{\theenumi}{\alph{enumi}}
\item The map \(v\) belongs to $C\left(\left[0,T\right]\times \R^d\right)\cap
C^{1+\alpha/2,2+\alpha}_{\loc}\left(\left[0,T\right)\times \R^d\right)$ and is a classical solution of
\begin{equation}\label{eq:PDE_v}
\left\{
\begin{aligned}
&\partial_t v
+ \sum_{i,j=1}^d A_{ij}(t,x)\partial_{ij} v
+ \sum_{i=1}^d b_i(t,x)\partial_i v
+ f\left(t,x,v,(\nabla_x v)\sigma\right) = 0,
\\
&v(T,x) = g(x),
\end{aligned}
\right.
\end{equation}
on \(\left[0,T\right)\times \R^d\).
\item Set \(u:=\nabla_x v\). For every \((s,x)\in[0,T)\times\R^d\),
\begin{equation*}
Z_t^{s,x} = u\bigl(t,X_t^{s,x}\bigr)\,\sigma\bigl(t,X_t^{s,x}\bigr)
\qquad \text{for } \Id t\otimes\Id\mathbb{P}\text{-almost every }
(t,\omega)\in[s,T]\times\Omega.
\end{equation*}
Consequently, \(w=u\sigma\) Lebesgue-almost everywhere on
\([0,T)\times\R^d\), and
\begin{equation*}
\sup_{(t,x)\in[0,T)\times\R^d}\abs{u(t,x)\,\sigma(t,x)} \le C_Z.
\end{equation*}
\end{enumerate}
\end{prop}

\begin{proof}[Proof of Proposition \ref{prop:reg}]
We split the proof into three steps. We first reduce the problem to a
BSDE whose generator is globally Lipschitz in \(\left(y,z\right)\), then establish
the continuity and local classical regularity of the decoupling field,
and finally identify the $Z$-component of the Markovian solution.
\\~\\
\noindent\emph{Step 1: Reduction to a BSDE globally Lipschitz in \(\left(y,z\right)\).}
Define the truncation maps \(\rho_Y:\R\to \R\) and \(\rho_Z:\R^d\to \R^d\) by
\begin{equation*}
\rho_Y(y)
=
\begin{cases}
y, & \text{if } C_Y=\infty,
\\
\max(-C_Y,\min(y,C_Y)), & \text{if } C_Y<\infty,
\end{cases}
\end{equation*}
and

\begin{equation*}
\rho_Z(z)
=
\begin{cases}
z, & \text{if } C_Z=\infty,\\
z, & \text{if } C_Z<\infty \text{ and } \abs{z}\le C_Z,\\
C_Z\dfrac{z}{\abs{z}}, & \text{if } C_Z<\infty \text{ and } \abs{z}>C_Z.
\end{cases}
\end{equation*}
When $C_Z<\infty$, this is the metric projection of $\R^d$ onto the
closed ball $\{z\in\R^d:\abs{z}\le C_Z\}$. Consequently, $\rho_Z$ is
$1$-Lipschitz,
\begin{equation}\label{eq:rhoZ_properties}
\abs{\rho_Z(z)}\le \abs{z},
\qquad
\rho_Z(z)=z\quad\text{whenever }\abs{z}\le C_Z,
\end{equation}
and the definition is well posed also when $C_Z=0$.

Set
\begin{equation*}
\widetilde f(t,x,y,z)
=
f\left(t,x,\rho_Y(y),\rho_Z(z)\right),
\qquad
(t,x,y,z)\in [0,T]\times \R^d\times \R\times \R^d.
\end{equation*}
By \eqref{eq:holder_f_tx}, the truncated generator
\(\widetilde f\) is H\"older continuous in \(t\) and \(x\), with exponent
\(\alpha/2\) in time and \(\alpha\) in space, locally uniformly in
\(\left(y,z\right)\). More precisely, for every \(R>0\), there exists
\(K_{f,R}>0\) such that, for all \(t,s\in \left[0,T\right]\),
\(x,\xi\in \R^d\), \(y\in \R\), and \(z\in \R^d\) satisfying
\(\abs{y}+\abs{z}\le R\),
\begin{equation}\label{eq:holder_tildef_tx}
\abs{
\widetilde f(t,x,y,z)-\widetilde f(s,\xi,y,z)
}
\leq
K_{f,R}
\left(
\abs{t-s}^{\alpha/2}
+
\abs{x-\xi}^{\alpha}
\right).
\end{equation}

By Assumption~\ref{ass:backward}\ref{it:B2}, the derivatives
\(\partial_y f\) and \(\nabla_z f\) are bounded by \(K_f\) on the strip \(\Sc\). Since both truncation maps \(\rho_Y\) and \(\rho_Z\) are $1$-Lipschitz and take values in the corresponding components of \(\Sc\), it follows that \(\widetilde f\) is
globally Lipschitz in \(\left(y,z\right)\), uniformly in \(\left(t,x\right)\). Moreover, by Remark~\ref{rem:CYCZ},
\(\abs{Y_t^{s,x}}\le C_Y\) for every \(t\in[s,T]\), almost surely, while
\(\abs{Z_t^{s,x}}\le C_Z\) for \(\Id t\otimes\Id\mathbb P\)-almost every
\((t,\omega)\in[s,T]\times\Omega\). Consequently,
\begin{equation*}
\widetilde f\left(t,X_t^{s,x},Y_t^{s,x},Z_t^{s,x}\right)
=
f\left(t,X_t^{s,x},Y_t^{s,x},Z_t^{s,x}\right)
\end{equation*}
for \(\Id t\otimes\Id\mathbb P\)-almost every
\((t,\omega)\in[s,T]\times\Omega\).
\\

We next record two structural consequences of
Assumption~\ref{ass:backward}\ref{it:B3}. First, by
\eqref{eq:struct_fx_zero}, the map \(x\mapsto f\left(t,x,y,0\right)\) is constant
whenever \(\abs{y}\le C_Y\). Hence
\begin{equation*}
\widetilde f(t,x,y,0)=\widetilde f(t,0,y,0).
\end{equation*}
Second, by \eqref{eq:struct_fx_lip},
\begin{equation*}
\abs{\nabla_x \widetilde f(t,x,y,z)}
=
\abs{\nabla_x f\left(t,x,\rho_Y(y),\rho_Z(z)\right)}
\le K_f \abs{\rho_Z(z)}
\le K_f \abs{z}.
\end{equation*}
Therefore, for all \(x,x'\in \R^d\),
\begin{equation}\label{eq:tildef_x_control}
\abs{\widetilde f(t,x,y,z)-\widetilde f(t,x',y,z)}
\le K_f \abs{z}\abs{x-x'}.
\end{equation}
Using \eqref{eq:lip_fyz} together with the fact that
\(\widetilde f\left(t,x,y,0\right)=\widetilde f\left(t,0,y,0\right)\),
we obtain
\begin{equation*}
\abs{\widetilde f(t,x,y,z)}
\le
\abs{f(t,0,0,0)} + K_f\abs{y} + K_f\abs{z}.
\end{equation*}
By \eqref{eq:holder_f_tx}, the map \(t\mapsto f(t,0,0,0)\) is bounded
on \(\left[0,T\right]\). Thus, \(\widetilde f\) has linear growth,
uniformly in \(t\).
\\

Because \(b\) and \(\sigma\) have at most linear growth and \(g\) has
at most polynomial growth, standard moment estimates for the forward
SDE yield
\begin{equation*}
\E\left(\abs{g(X_T^{s,x})}^p\right) < \infty
\end{equation*}
for every \(p\ge 2\) and every
\(\left(s,x\right)\in \left[0,T\right]\times \R^d\). Since \(\widetilde f\)
coincides with \(f\) along the Markovian solution, the pair
\(\left(Y^{s,x},Z^{s,x}\right)\) solves the BSDE with terminal condition
\(g\left(X_T^{s,x}\right)\) and generator
\(\widetilde f\left(t,X_t^{s,x},y,z\right)\). Since \(\widetilde f\) is
globally Lipschitz in \(\left(y,z\right)\) with linear growth, and since
the integrability assumed in Definition~\ref{defn:markov_sol} places
\(\left(Y^{s,x},Z^{s,x}\right)\) in the uniqueness class of the standard
\(L^p\)-theory for Lipschitz BSDEs, the a priori estimates of that
theory yield
\begin{equation*}
\E\left(\sup_{t\in[s,T]} \abs{Y_t^{s,x}}^p\right) < \infty,
\qquad
\E\left(\left(\int_s^T \abs{Z_r^{s,x}}^2\Id r\right)^{p/2}\right) < \infty,
\end{equation*}
for every \(p\ge 2\) and every
\(\left(s,x\right)\in \left[0,T\right]\times \R^d\).
\\~\\
\noindent\emph{Step 2: Continuity and local classical regularity of the decoupling field.}
We claim that \(v\) is continuous on \(\left[0,T\right]\times \R^d\).
\\

Fix a compact set \(K\subset \left[0,T\right]\times \R^d\) and two points
\(\left(s,x\right),\left(s',x'\right)\in K\). By symmetry, it is enough to
consider the case \(s\le s'\). Set
\begin{equation*}
F_r^{s,x}
:=
\widetilde f\left(r,X_r^{s,x},Y_r^{s,x},Z_r^{s,x}\right),
\qquad r\in[s,T].
\end{equation*}
Taking expectations in the BSDE between \(s\) and \(s'\), and using
\(Y_s^{s,x}=v(s,x)\), gives
\begin{equation*}
v(s,x)
=
\E\left[Y_{s'}^{s,x}+\int_s^{s'}F_r^{s,x}\,\Id r\right].
\end{equation*}
Since \(Y_{s'}^{s',x'}=v(s',x')\), it follows that
\begin{equation}\label{eq:cont_v_split}
\abs{v(s,x)-v(s',x')}
\le
\left(\E\abs{Y_{s'}^{s,x}-Y_{s'}^{s',x'}}^2\right)^{1/2}
+
\E\int_s^{s'}\abs{F_r^{s,x}}\,\Id r.
\end{equation}
The second term is negligible as \(s'\downarrow s\). Indeed, the linear
growth estimate for \(\widetilde f\) obtained in Step~1 yields
\begin{align}
\E\int_s^{s'}\abs{F_r^{s,x}}\,\Id r
&\le
C\abs{s'-s}
+
K_f\abs{s'-s}\,
\E\left(\sup_{r\in[s,T]}\abs{Y_r^{s,x}}\right)
+
K_f\E\int_s^{s'}\abs{Z_r^{s,x}}\,\Id r
\notag\\
&\le
C_K\left(\abs{s'-s}+\abs{s'-s}^{1/2}\right),
\label{eq:short_interval_driver}
\end{align}
where in the last line we used Cauchy--Schwarz together with the
\(L^2\)-estimates of Step~1 and the standard forward moment estimates,
which are uniform when the initial point ranges over the compact set
\(K\).

We now compare the two BSDEs only on their common interval
\([s',T]\). The standard \(L^2\)-stability estimate for Lipschitz BSDEs
gives
\begin{align}
\E\abs{Y_{s'}^{s,x}-Y_{s'}^{s',x'}}^2
&\le C_K \left(
\E\left(\abs{g(X_T^{s,x})-g(X_T^{s',x'})}^2\right)
\right.
\notag\\
&\quad +
\left.
\E\int_{s'}^T
\abs{
\widetilde f\left(r,X_r^{s,x},Y_r^{s,x},Z_r^{s,x}\right)
-
\widetilde f\left(r,X_r^{s',x'},Y_r^{s,x},Z_r^{s,x}\right)
}^2\Id r
\right).
\label{eq:cont_v_est}
\end{align}
Here \(C_K\) may be chosen uniformly for
\((s,x),(s',x')\in K\), since the Lipschitz constant of
\(\widetilde f\) in \((y,z)\) is global.
\\
We first treat the terminal term. Let \(R>0\), and define
\begin{equation*}
\Omega_R
=
\left\{
\abs{X_T^{s,x}}\le R,\ \abs{X_T^{s',x'}}\le R
\right\}.
\end{equation*}
Since \(g\in C^1\left(\R^d\right)\), its gradient is bounded on the compact ball \(\overline{B_R}\), so that \(g\) is Lipschitz
on \(B_R\) with some constant \(L_R>0\).
Therefore,
\begin{align*}
\E\left(\abs{g(X_T^{s,x})-g(X_T^{s',x'})}^2\right)
&\le
L_R^2\E\left(\abs{X_T^{s,x}-X_T^{s',x'}}^2\mathbf 1_{\Omega_R}\right)
\\
&\quad+
\E\left(\abs{g(X_T^{s,x})-g(X_T^{s',x'})}^2\mathbf 1_{\Omega_R^c}\right).
\end{align*}
The first term tends to \(0\) as \(\left(s',x'\right)\to \left(s,x\right)\) by the standard
continuity estimate for the forward flow. For the second term, the
polynomial growth of \(g\) and the uniform moment bounds for
\(X^{s,x}\) on compact sets imply, by Cauchy--Schwarz and Markov's
inequality, that
\begin{equation*}
\sup_{(s,x),(s',x')\in K}
\E\left(\abs{g(X_T^{s,x})-g(X_T^{s',x'})}^2\mathbf 1_{\Omega_R^c}\right)
\longrightarrow 0
\qquad \text{as } R\to\infty.
\end{equation*}
Thus,
\begin{equation}\label{eq:terminal_term_cont}
\E\left(\abs{g(X_T^{s,x})-g(X_T^{s',x'})}^2\right)
\longrightarrow 0
\qquad \text{as } (s',x')\to (s,x).
\end{equation}
\\
For the generator term in \eqref{eq:cont_v_est}, estimate
\eqref{eq:tildef_x_control} yields, for $r\in\left[s\vee s',T\right]$,
\begin{equation*}
\abs{
\widetilde f\left(r,X_r^{s,x},Y_r^{s,x},Z_r^{s,x}\right)
-
\widetilde f\left(r,X_r^{s',x'},Y_r^{s,x},Z_r^{s,x}\right)
}
\le
K_f \abs{Z_r^{s,x}}\abs{X_r^{s,x}-X_r^{s',x'}}.
\end{equation*}
By a supremum bound and the Cauchy--Schwarz inequality,
\begin{align*}
&\E\int_{s\vee s'}^T
\abs{
\widetilde f\left(r,X_r^{s,x},Y_r^{s,x},Z_r^{s,x}\right)
-
\widetilde f\left(r,X_r^{s',x'},Y_r^{s,x},Z_r^{s,x}\right)
}^2\Id r
\\
&\qquad\qquad\qquad\qquad \le
K_f^2
\E\left(
\sup_{s\vee s'\le r\le T}\abs{X_r^{s,x}-X_r^{s',x'}}^2
\int_{s\vee s'}^T \abs{Z_r^{s,x}}^2\Id r
\right)
\\
&\qquad\qquad\qquad\qquad \le
K_f^2
\left(
\E\left(\sup_{s\vee s'\le r\le T}\abs{X_r^{s,x}-X_r^{s',x'}}^4\right)
\right)^{1/2}
\\
&\qquad\qquad\qquad\qquad\quad\times
\left(
\E\left(\int_{s\vee s'}^T \abs{Z_r^{s,x}}^2\Id r\right)^2\right)^{1/2}.
\end{align*}
The second factor is uniformly
bounded for \(\left(s,x\right)\in K\), by the \(L^p\)-estimate from Step~1 with $p=4$, while the first factor tends to \(0\) as
\(\left(s',x'\right)\to \left(s,x\right)\) by the standard continuity estimate for the forward flow. Therefore,
\begin{equation}\label{eq:generator_term_cont}
\E\int_{s\vee s'}^T
\abs{
\widetilde f\left(r,X_r^{s,x},Y_r^{s,x},Z_r^{s,x}\right)
-
\widetilde f\left(r,X_r^{s',x'},Y_r^{s,x},Z_r^{s,x}\right)
}^2\Id r
\longrightarrow 0
\end{equation}
as \(\left(s',x'\right)\to \left(s,x\right)\).
Combining \eqref{eq:cont_v_split},
\eqref{eq:short_interval_driver}, \eqref{eq:cont_v_est},
\eqref{eq:terminal_term_cont}, and \eqref{eq:generator_term_cont}, we
obtain
\begin{equation*}
\abs{v(s,x)-v(s',x')}\longrightarrow0
\qquad\text{as }(s',x')\to(s,x)
\end{equation*}
whenever \(s\le s'\). The case \(s'\le s\) follows by interchanging
\((s,x)\) and \((s',x')\). Hence \(v\) is continuous on \(K\), and
therefore on all of \(\left[0,T\right]\times\R^d\).
\\

Now fix \(R>0\) and consider the parabolic cylinder and its parabolic boundary
\begin{equation*}
Q_R=[0,T)\times B_R,
\qquad
\partial_p Q_R
=
\left([0,T)\times \partial B_R\right)
\cup
\left(\{T\}\times \overline{B_R}\right).
\end{equation*}
On \(Q_R\), consider the Dirichlet problem
\begin{equation}\label{eq:local_dirichlet_v}
\left\{
\begin{aligned}
& \partial_t \hat v
+ \sum_{i,j=1}^d A_{ij}(t,x)\partial_{ij}\hat v
+ \sum_{i=1}^d b_i(t,x)\partial_i\hat v
+ \widetilde f\left(t,x,\hat v,(\nabla_x \hat v)\sigma\right) = 0,
\\
&\hat v = v \text{ on } \partial_p Q_R.
\end{aligned}
\right.
\end{equation}
Since \(v\) is continuous on \(\partial_p Q_R\), the boundary condition
in \eqref{eq:local_dirichlet_v} is continuous. Moreover, by \ref{it:P1} and \ref{it:P2}, together with Assumption \ref{ass:flow} and Remark \ref{rem:forward_roles}, the coefficients \(A\) and
\(b\) are H\"older continuous in \(t\) and \(x\), with exponent
\(\alpha/2\) in time and \(\alpha\) in space, on bounded cylinders, and
\(A\) is uniformly elliptic. By \eqref{eq:holder_tildef_tx},
\(\widetilde f\) is H\"older continuous in \(t\) and \(x\), with the same
exponents, locally uniformly in \(\left(y,z\right)\), and it is globally
Lipschitz in \(\left(y,z\right)\). Therefore the classical local theory
for semilinear uniformly parabolic Dirichlet problems yields a solution
\(\hat v\) such that
\begin{equation*}
\hat v\in C\left(Q_R\cup \partial_p Q_R\right)
\quad\text{and}\quad
\hat v\in C^{1+\alpha/2,2+\alpha}\left(Q'\right)
\end{equation*}
for every \(Q'\Subset Q_R\) -- see Lady\v{z}enskaja, Solonnikov and
Ural'ceva \cite[Chapter~VI, Theorem~4.2]{LadSolUra67}. The corresponding
interior Schauder estimates may also be found in Friedman
\cite[Chapter~3, Theorem~5]{friedman} and Krylov
\cite[Theorem~8.11.1]{Krylov1996}.
\\

Let $\tau_R
=
\inf\left\{ t\ge s : \abs{X_t^{s,x}}\ge R \right\}\wedge T$, and define, for \(t\in \left[s,T\right]\),
\begin{equation*}
\hat Y_t^{s,x}
=
\hat v\left(t\wedge \tau_R,X_{t\wedge \tau_R}^{s,x}\right).
\end{equation*}
Applying It\^o's formula to \(\hat v\left(t,X_t^{s,x}\right)\) up to
\(\tau_R\wedge (T-\varepsilon)\), and then letting
\(\varepsilon\downarrow 0\), using the continuity of \(\hat v\) on
\(Q_R\cup \partial_p Q_R\), shows that \(\hat Y^{s,x}\) satisfies on
\(\left[s,\tau_R\right]\) the BSDE with generator \(\widetilde f\) and
terminal value
\begin{equation*}
\hat Y_{\tau_R}^{s,x}
=
\hat v\left(\tau_R,X_{\tau_R}^{s,x}\right)
=
v\left(\tau_R,X_{\tau_R}^{s,x}\right)
=
Y_{\tau_R}^{s,x},
\end{equation*}
because \(\hat v=v\) on the parabolic boundary. By uniqueness for BSDEs with generator globally Lipschitz in \((y,z)\) on
the random interval \([s,\tau_R]\), we obtain
\begin{equation*}
\hat Y_t^{s,x}=Y_t^{s,x} \quad \text{for } t\in [s,\tau_R],
\qquad\text{and}\qquad
\hat Z_t^{s,x}=Z_t^{s,x}
\quad \Id t\otimes\Id\mathbb{P}\text{-a.e.\ on } [s,\tau_R],
\end{equation*}
where
\begin{equation*}
\hat Z_t^{s,x}
:=\bigl(\nabla_x\hat v\,\sigma\bigr)(t,X_t^{s,x})
\end{equation*}
denotes the martingale integrand furnished by the It\^o computation above.
In particular, $\hat v(t,x)=v(t,x)$ on $[0,T)\times B_R$. Since \(R>0\)
was arbitrary, \(v\) is a classical solution of
\eqref{eq:PDE_v}, with \(\widetilde f\) in place of \(f\), on \(\left[0,T\right)\times \R^d\) and belongs to
$C^{1+\alpha/2,2+\alpha}_{\loc}\left(\left[0,T\right)\times \R^d\right)$.
\\~\\

\noindent\emph{Step 3: Identification of the \(Z\)-component.}
Fix \((s,x)\in[0,T)\times\R^d\). For \(t<\tau_R\) we have
\(X_t^{s,x}\in B_R\), where \(\hat v=v\) and hence \(\nabla_x\hat v=u\), so
the identification of Step~2 reads
\(Z_t^{s,x}=u(t,X_t^{s,x})\,\sigma(t,X_t^{s,x})\),
\(\Id t\otimes\Id\mathbb{P}\)-a.e.\ on \([s,\tau_R]\). Since \(\tau_R=T\)
for all \(R\) large enough almost surely (the paths of \(X^{s,x}\) being
bounded on \([s,T]\)), letting \(R\to\infty\) proves the first claim of
part~(b). Taking \(s=0\) and using the defining relation
\(Z_t^{0,x}=w(t,X_t^{0,x})\), which holds for
\(\Id t\otimes\Id\mathbb P\)-almost every \((t,\omega)\), together with the identification just proved, we obtain
\begin{equation*}
w(t,X_t^{0,x})=u(t,X_t^{0,x})\,\sigma(t,X_t^{0,x})
\qquad
\Id t\otimes\Id\mathbb P\text{-a.e.}
\end{equation*}
Since the law of \(X_t^{0,x}\) is equivalent to Lebesgue measure
for every \(t\in(0,T)\) by Aronson's bounds \cite{Aronson}, Fubini's theorem then
yields \(w=u\sigma\) Lebesgue-almost everywhere on \([0,T)\times\R^d\).
Finally, since \(u\sigma\) is continuous on \([0,T)\times\R^d\) by part~(a)
and coincides almost everywhere with \(w\),
\(\sup_{[0,T)\times\R^d}\abs{u\sigma}
=\esssup_{[0,T)\times\R^d}\abs{w}\le C_Z\).
Since also \(\abs{v}\le C_Y\), this and \eqref{eq:rhoZ_properties} give
\(\widetilde f(t,x,v,u\sigma)=f(t,x,v,u\sigma)\) on \([0,T)\times\R^d\),
so the equation obtained in Step~2 is \eqref{eq:PDE_v}.
\\~\\
This proves both assertions of Proposition~\ref{prop:reg}.
\end{proof}

\subsection{Extra regularity of the decoupling field}\label{sec:extra_reg}

Proposition~\ref{prop:reg} provides the basic link between the
Markovian BSDE and its semilinear PDE, but the rigidity argument requires
more regularity because unique continuation and backward uniqueness are
applied to the differentiated field \(u=\nabla_xv\). Under
Assumption~\ref{ass:extra_reg}, we first show that \(u\) is regular enough
for the differentiated equation to hold in the classical sense and derive
the structural differential inequality used in the unique continuation
argument. Under the stronger hypotheses of Assumption~\ref{ass:rigidity},
we then obtain uniform bounds for \(u\) and \(\nabla_xu\) on positive-time
slabs. These bounds provide the input needed for backward uniqueness.
\\

Since the operator appearing in the rigidity argument is slightly more
convenient to write in divergence form, we introduce (recalling \eqref{eq:A_def})
\begin{equation}\label{eq:P_op_def}
\Pc \phi
=
\partial_t \phi + \sum_{i,j=1}^d \partial_i\left( A_{ij}\partial_j \phi\right).
\end{equation}
The point of this notation is that once the differentiated system for
\(u\) is established, the lower-order terms can be collected on the
right-hand side and compared directly with \(\abs{u}\) and
\(\abs{\nabla_xu}\).

\begin{prop}[Extra regularity of the spatial gradient]\label{prop:extra_reg}
Assume \ref{ass:parabolic} and \ref{ass:extra_reg}. Let \(\left(v,w\right)\) be the Markovian solution fixed in Assumption~\ref{ass:backward} and set $u := \nabla_x v$. Then,
\begin{equation*}
u \in C^{1+\alpha/2,2+\alpha}_{\loc}\left([0,T)\times \R^d;\R^d\right),
\end{equation*}
and, for each \(k\in\{1,\dots,d\}\),
\begin{align}
\label{eq:PDE_u}
\partial_t u_k &+ \sum_{i,j=1}^d A_{ij}\partial_{ij}u_k
+ \sum_{i=1}^d b_i\partial_i u_k
+ \sum_{i=1}^d (\partial_k b_i)u_i
+ \sum_{i,j=1}^d (\partial_k A_{ij})\partial_i u_j\nonumber
\\
&+ \partial_{x_k}f\left(t,x,v,u\sigma\right)
+ \partial_y f\left(t,x,v,u\sigma\right)u_k\nonumber
\\
&+ \sum_{\ell=1}^d \partial_{z_\ell}f\left(t,x,v,u\sigma\right)
\left(
\sum_{i=1}^d \sigma_{i\ell}\partial_i u_k
+
\sum_{i=1}^d u_i\partial_k\sigma_{i\ell}
\right) =0
\end{align}
in the classical sense on \(\left[0,T\right)\times \R^d\). Moreover, there
exists \(C>0\) such that
\begin{equation}\label{eq:struct_ineq}
\abs{\Pc u(t,x)}
\leq
C\left( \abs{u(t,x)} + \abs{\nabla_xu(t,x)} \right),
\qquad
(t,x)\in [0,T)\times \R^d.
\end{equation}
\end{prop}

\begin{proof}
The argument proceeds in three steps. We first differentiate the
equation for \(v\) by means of difference quotients and use local
Schauder estimates to obtain the improved regularity of \(u\). We then
pass to the limit and identify the differentiated system. Finally, we
derive the structural inequality \eqref{eq:struct_ineq}.
\\~\\
\noindent\emph{Step 1: Differentiating the semilinear equation.}
By Proposition~\ref{prop:reg}, the decoupling field satisfies $v\in
C^{1+\alpha/2,2+\alpha}_{\loc}\left(\left[0,T\right)\times \R^d\right)$ and solves \eqref{eq:PDE_v} classically.
This already gives two spatial derivatives of \(v\) with H\"older
control, but for the rigidity argument we need one further degree of regularity for $u$, with Hölder control on its second spatial derivatives and on its time derivative. To justify this, we differentiate
\eqref{eq:PDE_v} using difference quotients.
\\
Fix three cylinders
\begin{equation*}
Q'\Subset Q''\Subset Q\Subset \left[0,T\right)\times \R^d,
\end{equation*}
and let \(k\in\{1,\dots,d\}\). Write
\begin{equation*}
Q=I\times U,
\end{equation*}
where \(I=(a,b)\) with \(0<a<b<T\) or \(I=[0,b)\) with \(0<b<T\), and
\(U\subset \R^d\) is open. For \(h\neq 0\)
sufficiently small, define
\begin{equation*}
\delta_h^k v(t,x)
=
\frac{v(t,x+h e_k)-v(t,x)}{h}.
\end{equation*}
Since \(Q''\Subset Q\), for \(\abs{h}\) small enough the function
\(\delta_h^k v\) is well defined on \(Q''\). Subtracting the equation
for \(v\left(t,x\right)\) from the equation for
\(v\left(t,x+h e_k\right)\) and dividing by \(h\), we obtain on \(Q''\)
\begin{equation}\label{eq:dq_equation}
\partial_t \delta_h^k v
+ \sum_{i,j=1}^d A_{ij}(t,x)\partial_{ij}\delta_h^k v
+ \sum_{i=1}^d b_i(t,x)\partial_i \delta_h^k v
= G_h.
\end{equation}
For brevity, set
\begin{equation*}
\Psi(t,x):=((\nabla_xv)\sigma)(t,x).
\end{equation*}
Then
\begin{align*}
G_h
&=
-\sum_{i,j=1}^d \delta_h^k A_{ij}(t,x)\partial_{ij}v(t,x+h e_k)
-\sum_{i=1}^d \delta_h^k b_i(t,x)\partial_i v(t,x+h e_k)
\\
&\quad
-\frac{1}{h}\Bigl[
f\bigl(t,x+h e_k,v(t,x+h e_k),\Psi(t,x+h e_k)\bigr)
-f\bigl(t,x,v(t,x),\Psi(t,x)\bigr)
\Bigr].
\end{align*}

The last term is handled by the mean value theorem in the variables
\(\left(x,y,z\right)\), applied along the segment joining
\(\Xi(t,x)\) and \(\Xi(t,x+he_k)\), where
\begin{equation*}
\Xi(t,x):=(t,x,v(t,x),\Psi(t,x)).
\end{equation*}
By the definition of \(C_Y\), the bound of
Proposition~\ref{prop:reg}(b), and the convexity of \(\Sc\) in the
\((y,z)\)-variables, these segments lie in \(\Sc\). The bracket thus
equals \(h\) times a linear combination of \(1\), \(\delta_h^k v\), and
\(\delta_h^k\bigl((\nabla_x v)\sigma\bigr)\), whose coefficients are the
segment averages of \(\partial_{x_k}f\), \(\partial_y f\), and
\(\nabla_z f\). For \((t,x)\in\overline{Q''}\) and \(\abs{h}\) small
enough, the segments stay in a fixed compact subset of \(\Sc\), so the
estimate \eqref{eq:holder_R2} of \ref{it:R2} applies. Since \(v\),
\(\nabla_x v\), and \((\nabla_x v)\sigma\) are locally H\"older
continuous with time exponents at least \(1/2\) for the latter two, the
averaged coefficients are bounded in
\(C^{\alpha/2,\alpha}\left(\overline{Q''}\right)\), uniformly for
\(\abs{h}\) small enough. Moreover, each difference quotient satisfies
\begin{equation*}
\delta_h^k\varphi(t,x)
=\int_0^1\partial_k\varphi(t,x+\theta he_k)\Id\theta.
\end{equation*}
It therefore inherits, uniformly in \(h\), the local H\"older regularity of
\(\partial_k\varphi\). For \(\varphi\in\{A,b\}\), this is supplied by
\ref{it:R1} and the Lipschitz continuity of \(\sigma\). For
\(\varphi\in\{v,(\nabla_x v)\sigma\}\), it follows from
\(v\in C^{1+\alpha/2,2+\alpha}_{\loc}\) together with
\begin{equation*}
\partial_k\bigl((\nabla_x v)\sigma\bigr)
=(\partial_k\nabla_x v)\,\sigma+(\nabla_x v)\,\partial_k\sigma,
\end{equation*}
where the last term uses \ref{it:R1} once more. Consequently, \(G_h\) is bounded
in \(C^{\alpha/2,\alpha}\left(\overline{Q''}\right)\), uniformly for
\(\abs{h}\) small enough.
\\~\\
\noindent\emph{Step 2: Passage to the limit and the differentiated system.}
We now convert the difference-quotient equation into genuine
regularity for \(u=\nabla_x v\). For \(\abs{h}\) small enough, the
function \(\delta_h^k v\) solves on \(Q''\) a linear parabolic equation
with H\"older coefficients and H\"older right-hand side, uniformly in
\(h\). Therefore, the interior Schauder theory for linear parabolic
equations on bounded cylinders, see Friedman
\cite[Chapter~3, Theorem~5]{friedman} and Krylov
\cite[Theorem~8.11.1]{Krylov1996}, yields a constant
\(C=C\left(Q'',Q',\alpha\right)\), independent of \(h\), such that
\begin{equation}\label{eq:dq_schauder_bound}
\norm{\delta_h^k v}_{C^{1+\alpha/2,2+\alpha}(Q')}
\leq
C
\left(
\norm{\delta_h^k v}_{C(\overline{Q''})}
+
\norm{G_h}_{C^{\alpha/2,\alpha}(\overline{Q''})}
\right).
\end{equation}
By the estimates obtained in Step~1, the right-hand side of
\eqref{eq:dq_schauder_bound} is bounded uniformly for \(\abs{h}\)
small enough, and hence there exists \(h_0>0\) such that
\begin{equation*}
\sup_{0<\abs{h}<h_0}
\norm{\delta_h^k v}_{C^{1+\alpha/2,2+\alpha}(Q')}
<\infty.
\end{equation*}

Since \(v\) is already \(C^1\) in the space variable, we know that
\(\delta_h^k v \longrightarrow \partial_k v\) locally uniformly on
\(Q'\) as \(h\to 0\). Combined with the uniform bound
\eqref{eq:dq_schauder_bound}, the Arzelà--Ascoli theorem yields
\(\delta_h^k v \longrightarrow \partial_k v\) in
\(C^{1+\beta/2,2+\beta}(Q')\) for every \(\beta\in(0,\alpha)\). In
particular, \(\partial_t\delta_h^k v\to\partial_t\partial_k v\) and
\(\partial_{ij}\delta_h^k v\to\partial_{ij}\partial_k v\) uniformly on
\(Q'\). By lower semicontinuity of the Hölder seminorms under uniform
convergence, the bound \eqref{eq:dq_schauder_bound} passes to the
limit, and
\begin{equation*}
\partial_k v\in C^{1+\alpha/2,2+\alpha}(Q').
\end{equation*}
As \(Q'\Subset Q''\Subset Q\Subset \left[0,T\right)\times \R^d\) and
\(k\) were arbitrary, we conclude that
\begin{equation*}
u=\nabla_x v\in C^{1+\alpha/2,2+\alpha}_{\loc}\left([0,T)\times \R^d;\R^d\right).
\end{equation*}

Once this regularity is available, one may pass to the limit in
\eqref{eq:dq_equation} as \(h\to 0\), which yields the differentiated
system \eqref{eq:PDE_u} in the classical sense.
\\~\\
\noindent\emph{Step 3: The structural inequality.}
We now derive \eqref{eq:struct_ineq}. By definition of \(\Pc\),
\begin{equation*}
\Pc u_k
=
\partial_t u_k
+
\sum_{i,j=1}^d A_{ij}\partial_{ij}u_k
+
\sum_{i,j=1}^d (\partial_i A_{ij})\partial_j u_k.
\end{equation*}
Combining this identity with \eqref{eq:PDE_u}, we see that \(\Pc u_k\)
is a linear combination of \(u\), \(\nabla_xu\), and the inhomogeneous term \(\partial_{x_k}f\left(t,x,v,u\sigma\right)\).

By Assumption \ref{ass:parabolic}, the
forward-coefficient terms in~\eqref{eq:PDE_u} are bounded on
\([0,T)\times\R^d\). By the definition of \(C_Y\) and the bound of
Proposition~\ref{prop:reg}(b), one has \(\Xi(t,x)\in\Sc\) for every
\((t,x)\in[0,T)\times\R^d\). Since \eqref{eq:lip_fyz} holds on \(\Sc\), the
bound
\(\abs{\partial_y f(t,x,v,u\sigma)}+\abs{\nabla_z f(t,x,v,u\sigma)}\le K_f\)
holds uniformly. For the remaining term, the structural condition \ref{it:B3} gives
\begin{equation*}
\abs{\partial_{x_k}f(t,x,v,u\sigma)}
=
\abs{\partial_{x_k}f(t,x,v,u\sigma)-\partial_{x_k}f(t,x,v,0)}
\le K_f\abs{u\sigma}
\le C\abs{u},
\end{equation*}
where we used~\eqref{eq:struct_fx_zero}, \eqref{eq:struct_fx_lip}, and
boundedness of \(\sigma\). Hence
\begin{equation*}
\abs{\Pc u_k(t,x)}
\leq
C\left( \abs{u(t,x)} + \abs{\nabla_xu(t,x)} \right),
\qquad
(t,x)\in [0,T)\times \R^d.
\end{equation*}
Summing over \(k\) yields \eqref{eq:struct_ineq}, which completes the proof.
\end{proof}

The next result complements Proposition~\ref{prop:extra_reg}. The
previous proposition delivers pointwise regularity and the structural
inequality required by the rigidity argument, but the bound on
\(\Pc u\) it provides is expressed in terms of \(u\) and \(\nabla_x u\)
themselves. For backward uniqueness, one further needs uniform control
of \(u\) and \(\nabla_x u\) on slabs of the form \([t_1,t_2]\times \R^d\)
with \(0<t_1<t_2<T\). This is what we now establish under the stronger
coefficient assumptions gathered in Assumption~\ref{ass:rigidity}.

\begin{lem}[Differentiability with respect to the initial state]\label{lem:bsde_diff}
Assume \ref{ass:rigidity}. Fix $(s,x)\in[0,T)\times\R^d$. For each
coordinate vector $e_j$, there exist adapted processes
$(U^{(j)},V^{(j)})$ such that
\begin{multline}\label{eq:bsde_diffquot_conv}
\E\left[\sup_{t\in[s,T]}
\left|\frac{Y_t^{s,x+he_j}-Y_t^{s,x}}{h}-U_t^{(j)}\right|^2\right]
\\
+
\E\left[\int_s^T
\left|\frac{Z_t^{s,x+he_j}-Z_t^{s,x}}{h}-V_t^{(j)}\right|^2\,\Id t\right]
\longrightarrow0
\qquad\text{as }h\to0.
\end{multline}
Writing
$\Theta_r=(X_r^{s,x},Y_r^{s,x},Z_r^{s,x})$ and
$J_{s,t}=\nabla_xX_t^{s,x}$, and collecting the coordinate limits into the
matrix-valued processes $U_t^{s,x}=\nabla_xY_t^{s,x}$ and
$V_t^{s,x}=\nabla_xZ_t^{s,x}$, one has
\begin{equation}\label{eq:UV_integrability}
\E\left[\sup_{t\in[s,T]}|U_t^{s,x}|^2
+\int_s^T|V_t^{s,x}|^2\,\Id t\right]<\infty,
\end{equation}
and $(U^{s,x},V^{s,x})$ is the unique adapted pair satisfying
\eqref{eq:UV_integrability} and the variational BSDE
\begin{multline*}
U_t^{s,x}
=
\nabla_xg(X_T^{s,x})J_{s,T}
-
\int_t^T V_r^{s,x}\,\Id W_r
\\
+
\int_t^T
\Bigl[
\nabla_x f(r,\Theta_r)J_{s,r}
+
\partial_y f(r,\Theta_r)U_r^{s,x}
+
\nabla_z f(r,\Theta_r)V_r^{s,x}
\Bigr] \,\Id r.
\end{multline*}
\end{lem}

\begin{proof}
We give the argument because the spatial derivative
$\nabla_xf$ is not assumed to be globally bounded, so that a direct
application of a parameter-differentiability theorem would require
additional hypotheses not used elsewhere in the paper. Fix a coordinate
vector $e_j$ and, for $h\ne0$ sufficiently small, set
\begin{equation*}
\Delta_hX:=\frac{X^{s,x+he_j}-X^{s,x}}{h},\qquad
\Delta_hY:=\frac{Y^{s,x+he_j}-Y^{s,x}}{h},\qquad
\Delta_hZ:=\frac{Z^{s,x+he_j}-Z^{s,x}}{h}.
\end{equation*}
By Lemma~\ref{lem:jacobian_flow} and the standard differentiability of the
stochastic flow,
\begin{equation}\label{eq:diffquot_X}
\E\left[\sup_{r\in[s,T]}
\left|\Delta_hX_r-J_{s,r}e_j\right|^q\right]
\longrightarrow0
\qquad\text{for every }q\ge2.
\end{equation}
Let
$\Theta_r^h=(X_r^{s,x+he_j},Y_r^{s,x+he_j},Z_r^{s,x+he_j})$.
Since $f$ is $C^1$ in $(x,y,z)$, the mean-value formula gives
\begin{equation}\label{eq:diffquot_bsde}
\Delta_hY_t
=
\Delta_hg
+
\int_t^T
\bigl(F_r^h+a_r^h\Delta_hY_r+c_r^h\Delta_hZ_r\bigr)\,\Id r
-
\int_t^T\Delta_hZ_r\,\Id W_r,
\end{equation}
where
\begin{align*}
\Delta_hg
&=
\int_0^1
\nabla_xg\!\left(X_T^{s,x}+\theta(X_T^{s,x+he_j}-X_T^{s,x})\right)
\Delta_hX_T\,\Id\theta,\\
F_r^h
&=
\int_0^1
\nabla_xf\!\left(r,\Theta_r+\theta(\Theta_r^h-\Theta_r)\right)
\Delta_hX_r\,\Id\theta,\\
a_r^h
&=
\int_0^1
\partial_yf\!\left(r,\Theta_r+\theta(\Theta_r^h-\Theta_r)\right)\,\Id\theta,\\
c_r^h
&=
\int_0^1
\nabla_zf\!\left(r,\Theta_r+\theta(\Theta_r^h-\Theta_r)\right)\,\Id\theta.
\end{align*}
Here and below the notation
$\Theta_r+\theta(\Theta_r^h-\Theta_r)$ is understood componentwise.
Assumption~\ref{ass:rigidity}\ref{it:U2} gives
$|a_r^h|+|c_r^h|\le C_f$. Moreover, by
Assumption~\ref{ass:rigidity}\ref{it:U3},
\begin{equation}\label{eq:Fh_bound}
|F_r^h|
\le
C_f\bigl(|Z_r^{s,x}|+|Z_r^{s,x+he_j}|\bigr)|\Delta_hX_r|.
\end{equation}

We next record explicitly how the dependence of the generator on the
forward state is handled. Assumption~\ref{ass:rigidity}\ref{it:U3}
implies, for every $(r,x,y,z)$,
\begin{equation*}
|\nabla_x f(r,x,y,z)|
\le
|\nabla_x f(r,x,y,z)-\nabla_x f(r,x,y,0)|
\le C_f|z|.
\end{equation*}
Consequently, the fundamental theorem of calculus along the segment
joining $x$ and $x'$ gives
\begin{equation}\label{eq:f_x_stability}
|f(r,x',y,z)-f(r,x,y,z)|
\le C_f|z|\,|x'-x|.
\end{equation}
Together with the global Lipschitz bound in $(y,z)$ from
\ref{it:U2} and the boundedness of $\nabla_x g$ from \ref{it:U1}, the
standard $L^2$ stability estimate for Lipschitz BSDEs therefore yields
\begin{align*}
&\E\left[\sup_{r\in[s,T]}
\left|Y_r^{s,x+he_j}-Y_r^{s,x}\right|^2\right]
+
\E\left[\int_s^T
\left|Z_r^{s,x+he_j}-Z_r^{s,x}\right|^2\,\Id r\right]
\\
&\qquad\qquad\qquad\qquad\le
C\,\E\left|X_T^{s,x+he_j}-X_T^{s,x}\right|^2
+
C\,\E\left(
\int_s^T
|Z_r^{s,x}|\,
|X_r^{s,x+he_j}-X_r^{s,x}|\,\Id r
\right)^2.
\end{align*}
Since
$X^{s,x+he_j}-X^{s,x}=h\Delta_hX$, the first term on the right tends
to zero by \eqref{eq:diffquot_X}. For the second one, Cauchy--Schwarz
gives
\begin{multline*}
\E\left(
\int_s^T
|Z_r^{s,x}|\,
|X_r^{s,x+he_j}-X_r^{s,x}|\,\Id r
\right)^2
\\
\le
T|h|^2
\left(
\E\sup_{r\in[s,T]}|\Delta_hX_r|^4
\right)^{1/2}
\left(
\E\left(\int_s^T|Z_r^{s,x}|^2\,\Id r\right)^2
\right)^{1/2},
\end{multline*}
which also tends to zero by \eqref{eq:diffquot_X} with $q=4$ and the
$L^4$ estimate for $Z$ established in Proposition~\ref{prop:reg}.
Hence
\begin{equation}\label{eq:YZ_state_stability}
\E\left[\sup_{r\in[s,T]}
\left|Y_r^{s,x+he_j}-Y_r^{s,x}\right|^2\right]
+
\E\left[\int_s^T
\left|Z_r^{s,x+he_j}-Z_r^{s,x}\right|^2\,\Id r\right]
\longrightarrow0.
\end{equation}
In addition, the $L^q$ estimate of \cite[Theorem~4.2]{BDHPS}, the flow
moment bounds \eqref{eq:flow_moments}, and \eqref{eq:Fh_bound} show that,
for $h$ in a fixed neighbourhood of zero, the terminal variables
$\Delta_hg$ and the sources $F^h$ are uniformly square integrable. Hence
the linear BSDE estimate applied to \eqref{eq:diffquot_bsde} gives
\begin{equation}\label{eq:diffquot_uniform}
\sup_{0<|h|\le h_0}
\left(
\E\sup_{r\in[s,T]}|\Delta_hY_r|^2
+
\E\int_s^T|\Delta_hZ_r|^2\,\Id r
\right)<\infty.
\end{equation}

Let $(U^{(j)},V^{(j)})$ denote the unique solution of the linear BSDE
obtained by replacing in \eqref{eq:diffquot_bsde} $\Delta_hg,\ F^h,\
a^h,\ c^h$ respectively by
\begin{equation*}
\nabla_xg(X_T^{s,x})J_{s,T}e_j,
\quad
\nabla_xf(r,\Theta_r)J_{s,r}e_j,
\quad
\partial_yf(r,\Theta_r),
\quad
\nabla_zf(r,\Theta_r).
\end{equation*}
This limiting linear BSDE admits a unique adapted solution satisfying
\begin{equation*}
\E\left[\sup_{r\in[s,T]}|U_r^{(j)}|^2
+\int_s^T|V_r^{(j)}|^2\,\Id r\right]<\infty.
\end{equation*}
Indeed, $\partial_yf$ and $\nabla_zf$ are bounded, while
$|\nabla_xf(r,\Theta_r)|\le C_f|Z_r^{s,x}|$ by
\ref{it:U3}, and the latter source is square integrable after multiplication
by $J_{s,r}$ by \eqref{eq:flow_moments} and the $L^q$ estimates for $Z$.
We now justify the passage to the limit in the coefficients and in the
source. This is the only point where the local character of
Assumption~\ref{ass:extra_reg}\ref{it:R2} has to be taken into account.
Let
\begin{equation*}
\mu(\Id r,\Id\omega):=\Id r\,\mathbb{P}(\Id\omega)
\end{equation*}
be the product measure on $[s,T]\times\Omega$. The convergence
\eqref{eq:diffquot_X} and the preceding BSDE stability imply
\begin{equation}\label{eq:Theta_h_measure}
(X^{s,x+he_j},Y^{s,x+he_j},Z^{s,x+he_j})
\longrightarrow
(X^{s,x},Y^{s,x},Z^{s,x})
\qquad\text{in }\mu\text{-measure}.
\end{equation}
For $R>0$, introduce
\begin{equation*}
\Gamma_R^h
:=
\bigl\{(r,\omega):
|X_r^{s,x}|\vee |X_r^{s,x+he_j}|
\vee |Y_r^{s,x}|\vee |Y_r^{s,x+he_j}|
\vee |Z_r^{s,x}|\vee |Z_r^{s,x+he_j}|
\le R\bigr\}.
\end{equation*}
On $\Gamma_R^h$, every segment joining $\Theta_r$ and $\Theta_r^h$
lies in a fixed compact subset $K_R\subset\Sc$: the $x$--coordinates
remain in a fixed ball, while the $y$-- and $z$--coordinates remain in
the solution strip by convexity. Hence \eqref{eq:holder_R2} applies on
$K_R$. Moreover, the flow moment bounds and the $L^q$ estimates for the
BSDE, uniformly for $x+he_j$ with $|h|\le h_0$, give
\begin{equation}\label{eq:Gamma_tail}
\lim_{R\to\infty}
\sup_{0<|h|\le h_0}
\mu\bigl((\Gamma_R^h)^c\bigr)=0.
\end{equation}
It follows from \eqref{eq:Theta_h_measure}--\eqref{eq:Gamma_tail} and
\ref{it:R2} that the segment averages defining $a^h$ and $c^h$ converge
in $\mu$-measure to
$\partial_yf(\cdot,\Theta)$ and $\nabla_zf(\cdot,\Theta)$, respectively.
Since these coefficients are uniformly bounded by
Assumption~\ref{ass:rigidity}\ref{it:U2}, we in fact have
\begin{equation}\label{eq:ac_convergence}
a^h\longrightarrow \partial_y f(\cdot,\Theta)
\quad\text{and}\quad
c^h\longrightarrow \nabla_z f(\cdot,\Theta)
\qquad\text{in }L^p(\mu)
\end{equation}
for every finite $p\ge1$.

The source term requires a little more care because $\nabla_xf$ is not
bounded. Set
\begin{equation*}
\overline{\nabla_x f}^{\,h}_r
:=
\int_0^1
\nabla_xf\!\left(r,\Theta_r+\theta(\Theta_r^h-\Theta_r)\right)\,\Id\theta,
\qquad
F_r^h=\overline{\nabla_x f}^{\,h}_r\,\Delta_hX_r.
\end{equation*}
Write
\begin{align*}
X_r^{h,\theta}&:=X_r^{s,x}+\theta(X_r^{s,x+he_j}-X_r^{s,x}),\\
Y_r^{h,\theta}&:=Y_r^{s,x}+\theta(Y_r^{s,x+he_j}-Y_r^{s,x}),\\
Z_r^{h,\theta}&:=Z_r^{s,x}+\theta(Z_r^{s,x+he_j}-Z_r^{s,x}).
\end{align*}
Then
\begin{equation*}
\overline{\nabla_x f}^{\,h}_r-\nabla_xf(r,\Theta_r)=A_r^h+B_r^h,
\end{equation*}
where
\begin{align*}
A_r^h
&:=\int_0^1\Bigl[
\nabla_xf(r,X_r^{h,\theta},Y_r^{h,\theta},Z_r^{h,\theta})
-\nabla_xf(r,X_r^{h,\theta},Y_r^{h,\theta},Z_r^{s,x})
\Bigr] \,\Id\theta,\\
B_r^h
&:=\int_0^1\Bigl[
\nabla_xf(r,X_r^{h,\theta},Y_r^{h,\theta},Z_r^{s,x})
-\nabla_xf(r,X_r^{s,x},Y_r^{s,x},Z_r^{s,x})
\Bigr] \,\Id\theta.
\end{align*}
By Assumption~\ref{ass:rigidity}\ref{it:U3},
\begin{equation}\label{eq:Ah_bound}
|A_r^h|
\le C_f|Z_r^{s,x+he_j}-Z_r^{s,x}|,
\end{equation}
so that $A^h\to0$ in $L^2(\mu)$. For $B^h$, fix $R>0$ and restrict to
\begin{equation*}
\widetilde\Gamma_R^h
:=
\bigl\{(r,\omega):
|X_r^{s,x}|\vee|X_r^{s,x+he_j}|
\vee|Y_r^{s,x}|\vee|Y_r^{s,x+he_j}|
\vee|Z_r^{s,x}|\le R\bigr\}.
\end{equation*}
On this set the points entering $B^h$ remain in a fixed compact subset
of $\Sc$, and \ref{it:R2} gives
\begin{equation*}
|B_r^h|
\le C_R\left(
|X_r^{s,x+he_j}-X_r^{s,x}|^{\alpha}
+|Y_r^{s,x+he_j}-Y_r^{s,x}|^{\alpha}
\right).
\end{equation*}
Hence $B^h\mathbf 1_{\widetilde\Gamma_R^h}\to0$ in $L^2(\mu)$ by
\eqref{eq:Theta_h_measure}. On the complement, the structural condition
\ref{it:U3} and $\nabla_xf(t,x,y,0)=0$ give the deterministic bound
\begin{equation}\label{eq:Bh_bound}
|B_r^h|\le 2C_f|Z_r^{s,x}|.
\end{equation}
The higher-moment estimates for $X$ and $Y$, together with the
$L^q$ estimates for $Z$, imply by H\"older's inequality that the
$|Z|^2$--weighted contribution of
$(\widetilde\Gamma_R^h)^c$ tends to zero as first $h\to0$ and then
$R\to\infty$. Consequently,
\begin{equation}\label{eq:fxbar_convergence}
\overline{\nabla_x f}^{\,h}
\longrightarrow
\nabla_xf(\cdot,\Theta)
\qquad\text{in }L^2(\mu).
\end{equation}

We can now pass to the limit in $F^h$. Using
\begin{align*}
F_r^h-
\nabla_xf(r,\Theta_r)J_{s,r}e_j
&=
\left(\overline{\nabla_x f}^{\,h}_r-
\nabla_xf(r,\Theta_r)\right)\Delta_hX_r+
\nabla_xf(r,\Theta_r)
\left(\Delta_hX_r-J_{s,r}e_j\right),
\end{align*}
the second term tends to zero in $L^2(\mu)$ because
$|\nabla_xf(r,\Theta_r)|\le C_f|Z_r^{s,x}|$, the process $Z^{s,x}$
satisfies the $L^q$ estimates, and \eqref{eq:diffquot_X} holds for every
$q\ge2$. For the first term, truncate according to
$\sup_{r\in[s,T]}|\Delta_hX_r|\le M$. On this event,
\eqref{eq:fxbar_convergence} gives convergence to zero in $L^2(\mu)$.
On its complement, \ref{it:U3}, the uniform moment bounds for
$\Delta_hX$, and the $L^q$ estimates for $Z^{s,x}$ and
$Z^{s,x+he_j}$ make the squared products uniformly integrable. Letting
first $h\to0$, then $M\to\infty$, yields
\begin{equation}\label{eq:Fh_convergence}
\E\int_s^T
\left|F_r^h-
\nabla_xf(r,\Theta_r)J_{s,r}e_j\right|^2\,\Id r
\longrightarrow0.
\end{equation}

The terminal term is handled in the same way, but more simply. Since
$\nabla_xg$ is bounded and continuous by \ref{it:U1}, localization on
$\{|X_T^{s,x}|\vee|X_T^{s,x+he_j}|\le R\}$, followed by removal of the
localization using the flow moment bounds, gives
\begin{equation}\label{eq:Deltag_convergence}
\Delta_hg
\longrightarrow
\nabla_xg(X_T^{s,x})J_{s,T}e_j
\qquad\text{in }L^2(\Omega).
\end{equation}
Finally, since the differences in \eqref{eq:ac_convergence} are
uniformly bounded and $(U^{(j)},V^{(j)})$ satisfies the
square-integrability estimate above, we have
\begin{align*}
\E\int_s^T
|a_r^h-\partial_yf(r,\Theta_r)|^2|U_r^{(j)}|^2\,\Id r
&\longrightarrow0,
\end{align*}
as well as
\begin{align*}
\E\int_s^T
|c_r^h-\nabla_zf(r,\Theta_r)|^2|V_r^{(j)}|^2\,\Id r
&\longrightarrow0.
\end{align*}
Thus \eqref{eq:Fh_convergence}, \eqref{eq:Deltag_convergence}, and the
last two limits are precisely the data convergence required by the
standard $L^2$ stability estimate for linear BSDEs. Therefore
\begin{equation*}
\E\left[\sup_{r\in[s,T]}|\Delta_hY_r-U_r^{(j)}|^2\right]
+
\E\left[\int_s^T|\Delta_hZ_r-V_r^{(j)}|^2\,\Id r\right]
\longrightarrow0.
\end{equation*}
This is precisely \eqref{eq:bsde_diffquot_conv}. Doing this for
$j=1,\ldots,d$ proves differentiability with respect to the initial state
in the stated sense. Collecting the coordinate derivatives gives the
asserted matrix-valued pair $(U^{s,x},V^{s,x})$, which satisfies
\eqref{eq:UV_integrability} and the displayed variational BSDE.
\end{proof}

\begin{prop}[Uniform bound on the spatial gradient]\label{prop:u_bound}
Assume \ref{ass:rigidity}. Let $(v,w)$ be the Markovian solution fixed in Assumption~\ref{ass:backward} and set $u := \nabla_x v$. Then there exists a constant $C>0$, depending only on the
structural data, such that
\begin{equation}\label{eq:u_bound}
\abs{u(t,x)}
\le
e^{C(\T-t)/2}\norm{\nabla_x g}_\infty,
\qquad
(t,x)\in [0,\T)\times \R^d.
\end{equation}
\end{prop}
 
\begin{proof}
Fix $(s,x)\in [0,\T)\times\R^d$. By Lemma~\ref{lem:jacobian_flow} the
Jacobian flow $J_{s,t}:=\nabla_x X^{s,x}_t$ of \eqref{eq:SDE} is invertible on $[s,\T]$ and satisfies \eqref{eq:flow_moments}.
 
\medskip
\noindent\emph{Step 1: Variational FBSDE.}
Set $U^{s,x}_t := \nabla_x Y^{s,x}_t$ and $V^{s,x}_t := \nabla_x Z^{s,x}_t$.
By Lemma~\ref{lem:bsde_diff}, these derivatives exist, satisfy
\eqref{eq:UV_integrability}, and $(U^{s,x},V^{s,x})$ is the unique
adapted pair with this integrability property that solves the linear FBSDE
\begin{equation}\label{eq:var_FBSDE}
U^{s,x}_t
=
\nabla_x g(X^{s,x}_\T)\,J_{s,\T}
+ \int_t^\T \mathcal G_r\Id r
- \int_t^\T V^{s,x}_r\Id W_r,
\qquad t\in [s,\T],
\end{equation}
with driver
\begin{equation*}
\mathcal G_r
=
\nabla_x f(r,\Xi_r)\,J_{s,r}
+ \partial_y f(r,\Xi_r)\,U^{s,x}_r
+ \nabla_z f(r,\Xi_r)\,V^{s,x}_r,
\end{equation*}
where $\Xi_r$ denotes the forward state together with the Markovian
solution along it,
\begin{equation*}
\Xi_r
:=
\left(
X^{s,x}_r,\,
v\left(r,X^{s,x}_r\right),\,
u\left(r,X^{s,x}_r\right)\sigma\left(r,X^{s,x}_r\right)
\right).
\end{equation*}
By Definition~\ref{defn:markov_sol} and Proposition~\ref{prop:reg}(b), the
second and third components of \(\Xi_r\) coincide with \(Y_r^{s,x}\) and,
for \(\Id r\otimes\Id\mathbb{P}\)-almost every \((r,\omega)\), with
\(Z_r^{s,x}\). Therefore, the driver \(\mathcal G\) is exactly the variational driver obtained in Lemma~\ref{lem:bsde_diff}. By Assumption~\ref{ass:rigidity}\ref{it:U2} the driver is Lipschitz in
$(U,V)$ with constant $C_f$. By
Assumption~\ref{ass:rigidity}\ref{it:U3},
$\abs{\nabla_x f(r,\Xi_r)}\le C_f\abs{Z^{s,x}_r}$, and the standard
$L^q$ estimate for Markovian BSDEs with polynomial-growth data
\cite[Theorem~4.2]{BDHPS} yields
$\E\bigl(\int_s^\T\abs{Z^{s,x}_r}^2\Id r\bigr)^{q/2}<\infty$ for every
$q\ge 2$. Combined with \eqref{eq:flow_moments}, both the source
$\mathcal G_r - \partial_y f(r,\Xi_r)U^{s,x}_r - \nabla_z f(r,\Xi_r)V^{s,x}_r$
and the terminal $\nabla_x g(X^{s,x}_\T)\,J_{s,\T}$ belong to $L^q$ for
every $q\ge 2$, using Assumption~\ref{ass:rigidity}\ref{it:U1} for the
terminal term. Applying \cite[Theorem~4.2]{BDHPS} once more to
\eqref{eq:var_FBSDE},
\begin{equation}\label{eq:UV_moments}
\E\sup_{t\in [s,\T]} \abs{U^{s,x}_t}^q
< \infty
\qquad\text{for every } q\ge 2.
\end{equation}

By Proposition~\ref{prop:reg}, $Y^{s,x}_t = v(t,X^{s,x}_t)$ on $[s,\T)$
with $v\in C^{1+\alpha/2,2+\alpha}_\loc$. The chain rule, justified by
the $C^1$ regularity of the flow, gives
\begin{equation}\label{eq:identification}
U^{s,x}_t = u(t,X^{s,x}_t)\,J_{s,t},
\qquad t\in [s,\T),\ \mathbb{P}\text{-a.s.}
\end{equation}
 
\noindent\emph{Step 2: Continuity at the terminal time.}
Being the first component of the solution of \eqref{eq:var_FBSDE},
$U^{s,x}$ has a.s.\ continuous sample paths on $[s,\T]$, and
\begin{equation*}
U^{s,x}_t
\xrightarrow{t\to\T^-}
U^{s,x}_\T
=
\nabla_x g(X^{s,x}_\T)\,J_{s,\T}
\qquad \mathbb{P}\text{-a.s.}
\end{equation*}
By Lemma~\ref{lem:jacobian_flow}, $J_{s,t}^{-1}\to J_{s,\T}^{-1}$ a.s.\ as
$t\to\T^-$. Combined with \eqref{eq:identification},
\begin{equation}\label{eq:path_along}
u(t,X^{s,x}_t)
=
U^{s,x}_t\,J_{s,t}^{-1}
\xrightarrow{t\to\T^-}
\nabla_x g(X^{s,x}_\T)
\qquad \mathbb{P}\text{-a.s.}
\end{equation}
 
\noindent\emph{Step 3: Submartingale on truncated intervals.}
Set, for $t\in[s,\T)$,
\begin{equation*}
\bar Y_t := u(t,X^{s,x}_t),
\qquad
\bar Z_t^{k\ell} := \sum_i \sigma_{i\ell}(t,X^{s,x}_t)\,\partial_i
u_k(t,X^{s,x}_t).
\end{equation*}
Both processes are well defined by Proposition~\ref{prop:extra_reg}. By \eqref{eq:identification},
$\bar Y_r = U^{s,x}_r J_{s,r}^{-1}$ on $[s,\T)$. Together with
\eqref{eq:UV_moments}, \eqref{eq:flow_moments}, and Cauchy--Schwarz,
\begin{equation}\label{eq:Ybar_moments}
\E\sup_{r\in [s,\T)} \abs{\bar Y_r}^q
\le
\Bigl(\E\sup_{t\in[s,\T]}\abs{U^{s,x}_t}^{2q}\Bigr)^{1/2}
\Bigl(\E\sup_{t\in[s,\T]}\norm{J_{s,t}^{-1}}^{2q}\Bigr)^{1/2}
< \infty,
\qquad q\ge 2.
\end{equation}
Applying It\^o's formula componentwise
to $u_k(t,X^{s,x}_t)$ and substituting \eqref{eq:PDE_u} yields, for each
$k$,
\begin{equation*}
\Id\bar Y_t^k = -G_t^k\Id t + \sum_\ell \bar Z_t^{k\ell}\Id W_t^\ell,
\qquad t\in [s,\T),
\end{equation*}
with
\begin{align*}
G_t^k
=\ &\sum_i (\partial_k b_i)\bar Y_t^i
+ \sum_{i,j} (\partial_k A_{ij})\,\partial_i u_j(t,X^{s,x}_t)
+ \partial_{x_k}f(t,\Xi_t)
+ \partial_y f(t,\Xi_t)\bar Y_t^k
\\
&+ \sum_\ell \partial_{z_\ell}f(t,\Xi_t)\!
\left[\,\bar Z_t^{k\ell}
+ \sum_i \bar Y_t^i\,\partial_k \sigma_{i\ell}(t,X^{s,x}_t) \right].
\end{align*}
By Assumption~\ref{ass:parabolic}\ref{it:P3} the coefficient $\sigma$ is uniformly
elliptic with $\norm{\sigma^{-1}}_\infty \le \lambda^{-1/2}$, and
$\partial_i u_k = \sum_\ell (\sigma^{-1})_{\ell i}\,\bar Z^{k\ell}$
gives $\abs{\partial_i u_k}\le \lambda^{-1/2}\abs{\bar Z}$. Since
$A$ is given by \eqref{eq:A_def}, the bounds on $\sigma$
(\ref{ass:parabolic}\ref{it:P1}) and on $\nabla_x b,\nabla_x\sigma$
(Assumption \ref{ass:flow}) render $\nabla_x b$, $\nabla_x A$,
$\nabla_x\sigma$ bounded. With \ref{ass:rigidity}\ref{it:U2}
(boundedness of $\partial_y f$ and $\nabla_z f$) and
\ref{ass:rigidity}\ref{it:U3}---which, the third component of $\Xi_t$
being $u(t,X^{s,x}_t)\,\sigma(t,X^{s,x}_t)$, gives
$\abs{\partial_{x_k}f(t,\Xi_t)} \le C_f\abs{u(t,X^{s,x}_t)\,\sigma(t,X^{s,x}_t)}
\le C_f\norm{\sigma}_\infty\abs{\bar Y_t}$ pointwise on $[s,\T)$, with
$C_f\norm{\sigma}_\infty<\infty$ by
\ref{ass:parabolic}\ref{it:P1}---there exists a structural
$C^*>0$ such that
\begin{equation}\label{eq:G_bound}
\abs{G_t}
\le
C^*\bigl(\abs{\bar Y_t}+\abs{\bar Z_t}\bigr),
\qquad t\in [s,\T).
\end{equation}
It\^o's formula on $\abs{\bar Y_t}^2$ gives
\begin{equation*}
\Id\abs{\bar Y_t}^2
=
-2\langle\bar Y_t, G_t\rangle\Id t
+ \abs{\bar Z_t}^2\Id t
+ 2\langle\bar Y_t, \bar Z_t\Id W_t\rangle.
\end{equation*}
Using \eqref{eq:G_bound} and Young's
$2C^*\abs{\bar Y}\abs{\bar Z}\le (C^*)^2\abs{\bar Y}^2 + \abs{\bar
Z}^2$,
\begin{equation*}
\Id\abs{\bar Y_t}^2
\ge
-C\abs{\bar Y_t}^2\Id t
+ 2\langle\bar Y_t,\bar Z_t\Id W_t\rangle,
\qquad
C := 2C^* + (C^*)^2,
\end{equation*}
so that $e^{Ct}\abs{\bar Y_t}^2$ is a non-negative continuous local
submartingale on $[s,\T)$. Fix $\varepsilon\in (0,\T-s)$ and let
$\tau_n := \inf\{t\ge s : \abs{\bar Y_t}+\abs{\bar Z_t}\ge n\}
\wedge (\T-\varepsilon)$. On $[s,\tau_n]$ the integrand
$\langle\bar Y,\bar Z\rangle$ is bounded, so the stopped process
$e^{C(t\wedge\tau_n)}\abs{\bar Y_{t\wedge\tau_n}}^2$ is a true
submartingale on $[s,\T-\varepsilon]$, and for every
$t\in [s,\T-\varepsilon]$,
\begin{equation}\label{eq:stopped_submart}
\E_t\!\left[e^{C((\T-\varepsilon)\wedge\tau_n)}
\abs{\bar Y_{(\T-\varepsilon)\wedge\tau_n}}^2\right]
\ge
e^{C(t\wedge\tau_n)}\abs{\bar Y_{t\wedge\tau_n}}^2.
\end{equation}
Now $u$ and $\nabla_x u$ are continuous on $[s,\T-\varepsilon]\times\R^d$
by Proposition~\ref{prop:extra_reg}, and the random set
$\{X^{s,x}_r : r\in [s,\T-\varepsilon]\}$ is a.s.\ compact, so
$\sup_{r\in [s,\T-\varepsilon]}(\abs{\bar Y_r}+\abs{\bar Z_r})<\infty$
a.s., whence $\tau_n\to \T-\varepsilon$ a.s.\ as $n\to\infty$. Passing
to the limit in \eqref{eq:stopped_submart} via dominated convergence,
with the $L^1$ dominator $e^{C\T}\sup_{r\in [s,\T)}\abs{\bar Y_r}^2$
furnished by \eqref{eq:Ybar_moments}, yields
\begin{equation}\label{eq:submart}
e^{Ct}\abs{\bar Y_t}^2
\le
\E_t\!\left[e^{C(\T-\varepsilon)}\abs{\bar Y_{\T-\varepsilon}}^2\right],
\qquad t\in [s,\T-\varepsilon].
\end{equation}
 
\noindent\emph{Step 4: Limit $\varepsilon\to 0$.}
By \eqref{eq:Ybar_moments} with $q=4$,
$\sup_{r\in [s,\T)}\E\abs{\bar Y_r}^4<\infty$, so
$\{\abs{\bar Y_r}^2\}_{r\in [s,\T)}$ is uniformly integrable and
\eqref{eq:path_along} upgrades to convergence in $L^2(\mathbb{P})$. Passing to
the limit $\varepsilon\to 0$ in \eqref{eq:submart},
\begin{equation*}
e^{Ct}\abs{\bar Y_t}^2
\le
e^{C\T}\,\E_t\!\left[\abs{\nabla_x g(X^{s,x}_\T)}^2\right]
\le
e^{C\T}\norm{\nabla_x g}_\infty^2
\qquad \mathbb{P}\text{-a.s.,}
\end{equation*}
for every $t\in [s,\T)$. The right-hand side is finite by
Assumption~\ref{ass:rigidity}\ref{it:U1}. Specialising to $t=s$, where
$\bar Y_s = u(s,x)$,
\begin{equation*}
\abs{u(s,x)}^2
\le
e^{C(\T-s)}\norm{\nabla_x g}_\infty^2.
\end{equation*}
The point $(s,x)\in [0,\T)\times\R^d$ being arbitrary,
\eqref{eq:u_bound} follows.
\end{proof}

\begin{prop}[Uniform gradient bound on positive-time slabs]\label{prop:grad_slab_bound}
Assume \ref{ass:rigidity}. Let $(v,w)$ be the Markovian solution fixed in Assumption~\ref{ass:backward} and set $u := \nabla_x v$. Then, for every $\varepsilon\in\left(0,T/2\right)$,
\begin{equation}\label{eq:grad_slab_bound}
\sup_{(t,x)\in [\varepsilon,T-\varepsilon]\times \R^d}
\abs{\nabla_x u(t,x)} < \infty.
\end{equation}
\end{prop}
 
\begin{proof}
We prove \eqref{eq:grad_slab_bound} in four steps.
 
\medskip
\noindent\emph{Step 1: Reduction to a linear system with bounded data.}
 
Set $M_0:=\sup_{(t,x)\in[0,T)\times\R^d}\abs{u(t,x)}$, which is finite by
Proposition~\ref{prop:u_bound}. By Proposition~\ref{prop:extra_reg},
$u\in C^{1+\alpha/2,2+\alpha}_{\loc}([0,T)\times\R^d;\R^d)$ and solves
\eqref{eq:PDE_u} classically on $[0,T)\times\R^d$.
 
For each $k$, the second-order part of \eqref{eq:PDE_u} is
$\sum_{i,j=1}^d A_{ij}\partial_{ij}u_k$, the same scalar operator on
every component. The principal part is therefore diagonal, and the
components couple only at lower order, the sole first-order inter-component coupling
being the term $\sum_{i,j}(\partial_k A_{ij})\partial_i u_j$. By Remark \ref{rem:forward_roles}, the matrix
$A$ is bounded, uniformly elliptic and globally Lipschitz on
$[0,T]\times\R^d$ --- hence uniformly continuous.
 
Each first- and zeroth-order coefficient of \eqref{eq:PDE_u} is bounded
on $[0,T)\times\R^d$: those formed from $b,\sigma,\nabla_x b,\nabla_x\sigma$
by \ref{ass:parabolic}\ref{it:P1} and Assumption \ref{ass:flow}, and
$\partial_y f(t,x,v,u\sigma)$, $\nabla_z f(t,x,v,u\sigma)$ by the bound
$\abs{\partial_y f}+\abs{\nabla_z f}\le C_f$ of
Assumption~\ref{ass:rigidity}\ref{it:U2}, valid at every value of $y$.
The right-hand side of \eqref{eq:PDE_u} is
$F:=-\nabla_x f(\cdot,\cdot,v,u\sigma)$. By the structural identity
$\nabla_x f(t,x,v,0)=0$ and the $z$-Lipschitz bound of $\nabla_x f$
(Assumption~\ref{ass:rigidity}\ref{it:U3}),
\begin{equation}\label{eq:source_bound}
\abs{F}
=\abs{\nabla_x f(t,x,v,u\sigma)-\nabla_x f(t,x,v,0)}
\le C_f\abs{u\sigma}
\le C_f\norm{\sigma}_\infty M_0,
\end{equation}
using $\abs{u}\le M_0$. Since $u\in C^{1+\alpha/2,2+\alpha}_{\loc}$, it
lies in the parabolic Sobolev space
$W^{2,1,p}_{\loc}([0,T)\times\R^d;\R^d)$, i.e.\ $\partial_t
u,\nabla_x u,\nabla_x^2 u\in L^p_{\loc}$, for every $p\in[1,\infty)$, and
is a strong solution of the linear, uniformly parabolic system
\eqref{eq:PDE_u}, with diagonal scalar principal part, bounded
lower-order coefficients, and right-hand side $F$ bounded as in
\eqref{eq:source_bound}.
 
\medskip
\noindent\emph{Step 2: Interior $L^p$ estimate on translated cylinders.}
 
Fix $p:=d+3$ and $\varepsilon\in\left(0,T/2\right)$. For $x_*\in\R^d$, set
the open cylinders
\begin{equation*}
Q'(x_*):=(\varepsilon,T-\varepsilon)\times B_1(x_*),
\qquad
Q(x_*):=\bigl(\tfrac{\varepsilon}{2},T-\tfrac{\varepsilon}{2}\bigr)\times B_2(x_*),
\end{equation*}
so that $Q'(x_*)\Subset Q(x_*)$ and $\overline{Q(x_*)}\subset[0,T)\times\R^d$.
By Step~1, $u\in C^{1+\alpha/2,2+\alpha}(\overline{Q(x_*)})$, hence
$u\in W^{2,1,p}(Q(x_*);\R^d)$, and $u$ solves on $Q(x_*)$ the system
\eqref{eq:PDE_u}, which is uniformly parabolic with diagonal scalar
principal part and bounded lower-order coefficients. The interior
$W^{2,1,p}$ estimate for such systems (Solonnikov~\cite[Theorem~5.7]{Solonnikov}) provides a constant
$C_\varepsilon>0$ such that
\begin{equation}\label{eq:Lp_estimate}
\norm{u}_{W^{2,1,p}(Q'(x_*))}
\le
C_\varepsilon\bigl(\norm{F}_{L^p(Q(x_*))}+\norm{u}_{L^p(Q(x_*))}\bigr).
\end{equation}
The constant $C_\varepsilon$ depends only on $d$, $p$, the ellipticity
constant and modulus of continuity of $A$, the bounds on the lower-order
coefficients, and the parabolic distance from $Q'(x_*)$ to
$\partial Q(x_*)$. These cylinders
are spatial translates of $Q'(0)$ and $Q(0)$, so their sizes and
parabolic separation are independent of $x_*$, and the remaining
quantities are global by Step~1. Hence $C_\varepsilon$ is independent of
$x_*$.
 
By the bounds $\abs{F}\le C_f\norm{\sigma}_\infty M_0$
(see~\eqref{eq:source_bound}) and $\abs{u}\le M_0$ of Step~1, and since
$\abs{Q(x_*)}=(T-\varepsilon)\abs{B_2}$ is independent of $x_*$,
\begin{equation*}
\norm{F}_{L^p(Q(x_*))}+\norm{u}_{L^p(Q(x_*))}
\le
\left(C_f\norm{\sigma}_\infty+1\right)M_0\,
\bigl((T-\varepsilon)\abs{B_2}\bigr)^{1/p}.
\end{equation*}
Together with \eqref{eq:Lp_estimate}, this yields
\begin{equation}\label{eq:Lp_uniform}
\sup_{x_*\in\R^d}\norm{u}_{W^{2,1,p}(Q'(x_*))}
\le
C_\varepsilon\left(C_f\norm{\sigma}_\infty+1\right)M_0
\bigl((T-\varepsilon)\abs{B_2}\bigr)^{1/p}
<\infty,
\end{equation}
where $\abs{B_2}$ is the volume of a ball of radius $2$ in $\R^d$.
 
\medskip
\noindent\emph{Step 3: Parabolic Sobolev embedding.}
 
Set $\beta:=1-(d+2)/p=\tfrac{1}{d+3}\in(0,1)$. As $p=d+3>d+2$, the
parabolic Sobolev embedding (\cite[Ch.~II, Lemma~3.3]{LadSolUra67}),
applied componentwise, provides a constant $C_S$ such that
\begin{equation}\label{eq:embedding}
\sup_{\overline{Q'(x_*)}}\abs{\nabla_x u}
\le
\norm{u}_{C^{(1+\beta)/2,\,1+\beta}(\overline{Q'(x_*)})}
\le
C_S\,\norm{u}_{W^{2,1,p}(Q'(x_*))},
\qquad x_*\in\R^d,
\end{equation}
the first inequality because that norm dominates $\sup\abs{\nabla_x u}$.
The constant $C_S$ depends only on $d$, $p$, and the geometry of
$Q'(x_*)$, hence not on $x_*$, the cylinders being translates of $Q'(0)$.
Taking the supremum over $x_*$ in \eqref{eq:embedding} and using
\eqref{eq:Lp_uniform},
\begin{equation}\label{eq:grad_uniform_cyl}
\sup_{x_*\in\R^d}\ \sup_{\overline{Q'(x_*)}}\abs{\nabla_x u}
\le
C_S\,\sup_{x_*\in\R^d}\norm{u}_{W^{2,1,p}(Q'(x_*))}
<\infty.
\end{equation}
 
\medskip
\noindent\emph{Step 4: Covering and conclusion.}
 
Every point $(t,x)\in[\varepsilon,T-\varepsilon]\times\R^d$ lies in
$\overline{Q'(x)}=[\varepsilon,T-\varepsilon]\times\overline{B_1(x)}$,
since $t\in[\varepsilon,T-\varepsilon]$ and $x\in\overline{B_1(x)}$.
Therefore
\begin{equation*}
\sup_{(t,x)\in[\varepsilon,T-\varepsilon]\times\R^d}\abs{\nabla_x u(t,x)}
\le
\sup_{x_*\in\R^d}\ \sup_{\overline{Q'(x_*)}}\abs{\nabla_x u}
<\infty,
\end{equation*}
the last inequality by \eqref{eq:grad_uniform_cyl}. This is
\eqref{eq:grad_slab_bound}. As $\varepsilon\in(0,T/2)$ was arbitrary, the
bound holds for every such $\varepsilon$.
\end{proof}

\begin{cor}[Uniform bounds on positive-time slabs]\label{cor:slab_bounds}
Under the assumptions of Proposition~\ref{prop:grad_slab_bound}, for every
\(\varepsilon\in\left(0,T/2\right)\),
\begin{equation}\label{eq:combined_slab_bound}
\sup_{(t,x)\in [\varepsilon,T-\varepsilon]\times \R^d}
\left( \abs{u(t,x)} + \abs{\nabla_x u(t,x)} \right) < \infty.
\end{equation}
\end{cor}

\begin{proof}
By Proposition~\ref{prop:u_bound}, \(\sup_{[0,T)\times \R^d}\abs{u}<\infty\),
hence \(\sup_{[\varepsilon,T-\varepsilon]\times \R^d}\abs{u}<\infty\). By
Proposition~\ref{prop:grad_slab_bound},
\(\sup_{[\varepsilon,T-\varepsilon]\times \R^d}\abs{\nabla_x u}<\infty\).
Adding the two bounds gives \eqref{eq:combined_slab_bound}.
\end{proof}

\subsection{Malliavin representation}\label{sec:malliavin_rep}

At this point the global estimate of Proposition~\ref{prop:u_bound} is
available, so the Malliavin factorization announced at the beginning of
Section~\ref{sec:regularity} can be justified under the full rigidity
assumptions.

\begin{lem}[Malliavin representation under the rigidity assumptions]\label{lem:malliavin_rep}
Assume \ref{ass:rigidity}. Then, for every $t\in[0,T)$,
$Y_t^{0,x_0}\in\Dcal^{1,2}$ and
\begin{equation}\label{eq:DY_formula}
D_rY_t^{0,x_0}
=
\nabla_xv(t,X_t^{0,x_0})D_rX_t^{0,x_0}\mathbf 1_{[0,t]}(r),
\qquad r\in[0,T].
\end{equation}
\end{lem}

\begin{proof}
By Assumption~\ref{ass:flow}, $X_t^{0,x_0}\in\Dcal^{1,2}$ for every
$t\in[0,T]$. See Nualart \cite[Theorem~2.2.1]{dNualart}. Fix $t<T$.
Proposition~\ref{prop:reg}(a) gives $v(t,\cdot)\in C^1(\R^d)$, while
Proposition~\ref{prop:u_bound} yields the global bound
\begin{equation*}
\sup_{x\in\R^d}|\nabla_xv(t,x)|<\infty.
\end{equation*}
Thus $v(t,\cdot)$ is continuously differentiable with bounded gradient, and the
Malliavin chain rule \cite[Proposition~1.2.3]{dNualart} applies directly to
$Y_t^{0,x_0}=v(t,X_t^{0,x_0})$: it gives
$Y_t^{0,x_0}\in\Dcal^{1,2}$ together with \eqref{eq:DY_formula}.
\end{proof}

\section{PDE rigidity for the differentiated decoupling field}\label{sec:rigidity}

We now come to the analytic core of the argument. The purpose of this
section is to exploit the structure obtained in
Section~\ref{sec:extra_reg} for the differentiated decoupling field
\begin{equation*}
u=\nabla_x v
\end{equation*}
and to turn it into a rigidity statement for its zero set. This is the
step that replaces the one-dimensional sign arguments of the
classical BSDE density literature.
\\

Throughout the section, we write
$u=\nabla_x v$ and recall \eqref{eq:A_def}. We rely on two facts established earlier. First,
under the assumptions of Proposition~\ref{prop:extra_reg}, the field
\(u\) solves a linear parabolic system and satisfies the structural
differential inequality
\begin{equation*}
\abs{\Pc u}\leq C\left(\abs{u}+\abs{\nabla_xu}\right).
\end{equation*}
Second, under the stronger hypotheses of Assumption \ref{ass:rigidity},
both \(u\) and \(\nabla_xu\) are uniformly bounded on every positive-time
slab (Corollary~\ref{cor:slab_bounds}). The role of the present section is to show that these two pieces
of information force the following rigidity property: at any fixed positive
time, the zero set of \(u\left(t,\cdot\right)\) cannot have positive Lebesgue
measure unless \(u\) vanishes identically from that time onward.
\\

The argument is organized in three steps. We first prove a unique
continuation statement from sets of positive measure at a fixed time.
We then establish a backward uniqueness lemma on the whole space.
Finally, we combine the two results to rule out large zero sets for
\(\nabla_x v\), which is the decisive input for the density theorem.

\subsection{Unique continuation from sets of positive measure}

We begin with the unique continuation step. Classical unique continuation results for parabolic equations propagate vanishing from a spatial open set at a fixed time to the entire spatial slice. This is insufficient for our purposes: since $X_t$ admits a density with respect to Lebesgue measure under Assumption~\ref{ass:parabolic}, the relevant configuration to exclude is that the zero set of $u(t,\cdot)$ has positive Lebesgue measure, rather than that it contains an open set. The lemma
below provides exactly this: it propagates vanishing from a set of positive Lebesgue measure to the entire spatial slice at the same time.

\begin{lem}[Unique continuation from a set of positive measure]\label{lem:uc}
Let $u \colon \left[0,T\right]\times \R^d \to \R^d$ belong to $C^{1,2}\left(\left(0,T\right)\times \R^d;\R^d\right)$. With Assumption \ref{ass:parabolic} in place, suppose there exists $C>0$ such that
\begin{equation}\label{eq:uc_ass}
\abs{\Pc u(t,x)} \leq C\left( \abs{u(t,x)} + \abs{\nabla_xu(t,x)} \right),
\qquad (t,x)\in (0,T)\times \R^d.
\end{equation}
If for some $t_0\in \left(0,T\right)$ the set $\{x\in \R^d : u\left(t_0,x\right)=0\}$ has positive Lebesgue measure, then $u\left(t_0,\cdot\right)\equiv 0$ on $\R^d$.
\end{lem}

\begin{proof}
The proof combines a density-point argument with quantitative local
estimates for \(u\), namely Caccioppoli-type bounds and a two-sphere
one-cylinder inequality.

We first make precise how the scalar estimates of
\cite{lEscauriaza2006} are used for the vector-valued function
\(u=(u^1,\dots,u^d)\). The operator \(\Pc\) acts componentwise, with the
same principal coefficients on every component. All the identities in
the scalar argument that are used to obtain the local energy,
frequency/doubling, and Carleman estimates are quadratic. For a
vector-valued function, one applies the corresponding identity to each
component and sums. Equivalently, in the frequency-function part of the
argument one replaces the scalar quadratic quantities by their sums over
the components. All integrations by parts remain componentwise, while
the Cauchy--Schwarz steps are simply the Cauchy--Schwarz inequality in
the Euclidean target space. Likewise, the Carleman estimate is summed
over the components before any differential inequality is invoked.

The crucial point is that the system inequality \eqref{eq:uc_ass} is used
only after this summation. Indeed, the source terms then combine as
\begin{equation}\label{eq:uc_vector_source}
\sum_{k=1}^d |\Pc u^k|^2
=
|\Pc u|^2
\leq
C^2\bigl(|u|+|\nabla_xu|\bigr)^2
\leq
2C^2\bigl(|u|^2+|\nabla_xu|^2\bigr).
\end{equation}
Consequently, the lower-order contributions are absorbed after summing,
exactly as in the scalar proof, with constants changed at most by a
factor depending on the constants already present in
\eqref{eq:uc_ass}. The resulting estimates are therefore the same
quadratic estimates with
\begin{equation*}
|u|^2=\sum_{k=1}^d|u^k|^2,
\qquad
|\nabla_xu|^2=\sum_{k=1}^d|\nabla_xu^k|^2.
\end{equation*}
This yields the vector-valued forms of
\cite[Equation~(2.20)]{lEscauriaza2006} and
\cite[Theorem~2]{lEscauriaza2006} used below. Notice in particular that
we never require, and do not claim, separate componentwise inequalities
of the form
\begin{equation*}
|\Pc u^k|\leq C\bigl(|u^k|+|\nabla_xu^k|\bigr).
\end{equation*}
By an affine change of the time variable, we may assume without loss of
generality that \(T=t_0+17\). This change rescales the coefficients of
\(\Pc\) and the constant in \eqref{eq:uc_ass} by a constant factor, so
all standing hypotheses are preserved, with constants now depending also
on \(t_0\), which is fixed throughout the proof. Let
\begin{equation*}
G  =  \{x\in \R^d : u(t_0,x)=0\}.
\end{equation*}
By assumption, \(G\) has positive Lebesgue measure. By the Lebesgue
density theorem (see \cite[Theorem~1.35]{lEvans2015}), almost every
point of \(G\) is a Lebesgue density point, that is, a point
\(\bar x\in G\) satisfying
\begin{equation*}
\lim_{r\downarrow 0}\frac{|B_r(\bar x)\cap G|}{|B_r(\bar x)|}=1.
\end{equation*}
Fix such an \(\bar x\). After translating the space variable, we may
assume without loss of generality that \(\bar x=0\), so that
\begin{equation}\label{eq:uc_density_point}
\lim_{r\downarrow 0} \frac{|B_r\cap G|}{|B_r|}=1.
\end{equation}
We divide the proof into three steps.
\\~\\
\noindent\emph{Step 1: A small-ball estimate near a density point.}
We argue by contradiction. Suppose first that
\begin{equation*}
\norm{u(t_0,\cdot)}_{L^2(B_1)}>0.
\end{equation*}
Fix \(\varepsilon>0\). By \eqref{eq:uc_density_point}, there exists
\(r\left(\varepsilon\right)\in \left(0,1/2\right)\) such that
\begin{equation}\label{eq:uc_density_estimate}
|B_r\setminus G|<\varepsilon |B_r|,
\qquad 0<r<r(\varepsilon).
\end{equation}
Since \(u\left(t_0,\cdot\right)=0\) on \(G\), we have, for every
\(r<r\left(\varepsilon\right)\),
\begin{equation*}
\norm{u(t_0,\cdot)}_{L^2(B_r)}
=
\norm{u(t_0,\cdot)}_{L^2(B_r\setminus G)}.
\end{equation*}

We now distinguish between the cases \(d\ge2\) and \(d=1\). Suppose first that
\(d\ge2\). Applying H\"older's inequality with exponents \(d\) and
\(d/\left(d-1\right)\), we obtain
\begin{equation}\label{eq:uc_holder}
\norm{u(t_0,\cdot)}_{L^2(B_r)}
\leq
C |B_r\setminus G|^{1/(2d)}
\norm{u(t_0,\cdot)}_{L^{2d/(d-1)}(B_r)}.
\end{equation}
We next estimate the \(L^{2d/\left(d-1\right)}\)-norm by the \(L^2\)-norm of the
gradient. Since \(u\left(t_0,\cdot\right)\) vanishes on a subset of \(B_r\) with
measure at least \(|B_r|/2\) when \(r<r\left(1/2\right)\), Lemma~7.16 in
Gilbarg--Trudinger \cite{dGilbarg2001} gives, for each component
\(u^k\) of \(u\),
\begin{equation*}
|u^k(t_0,x)|
\leq
C \int_{B_r} |x-y|^{1-d} |\nabla_xu^k(t_0,y)| \mathrm{d}y,
\qquad \text{for a.e. } x\in B_r.
\end{equation*}
Combining this with Lemma~7.12 in \cite{dGilbarg2001} (with $\mu=1/d$) yields
\begin{equation}\label{eq:uc_sobolev_component}
\norm{u^k(t_0,\cdot)}_{L^{2d/(d-1)}(B_r)}
\leq
C r^{1/2}
\norm{\nabla_xu^k(t_0,\cdot)}_{L^2(B_r)}.
\end{equation}
Summing over \(k=1,\dots,d\), we infer
\begin{equation}\label{eq:uc_sobolev_vector}
\norm{u(t_0,\cdot)}_{L^{2d/(d-1)}(B_r)}
\leq
C r^{1/2}
\norm{\nabla_xu(t_0,\cdot)}_{L^2(B_r)}.
\end{equation}
Using \eqref{eq:uc_density_estimate} and \(|B_r|\asymp r^d\), we therefore obtain
\begin{equation}\label{eq:uc_pre_caccioppoli_highd}
\norm{u(t_0,\cdot)}_{L^2(B_r)}
\leq
C\varepsilon^{1/(2d)}r\,
\norm{\nabla_xu(t_0,\cdot)}_{L^2(B_r)}.
\end{equation}

Suppose now that \(d=1\). Then \(B_r=(-r,r)\), and, for
\(r<r(1/2)\), the set \(G\cap B_r\) has positive measure and is therefore
nonempty. Choose \(y_r\in G\cap B_r\). Since
\(u(t_0,y_r)=0\), the fundamental theorem of calculus and
Cauchy--Schwarz give, for every \(x\in B_r\),
\begin{equation*}
|u(t_0,x)|
=
\left|\int_{y_r}^x \partial_xu(t_0,z)\,\Id z\right|
\leq
(2r)^{1/2}
\norm{\partial_xu(t_0,\cdot)}_{L^2(B_r)}.
\end{equation*}
Consequently, using again that \(u(t_0,\cdot)=0\) on \(G\) and
\eqref{eq:uc_density_estimate},
\begin{align}\label{eq:uc_pre_caccioppoli_oned}
\norm{u(t_0,\cdot)}_{L^2(B_r)}
=
\norm{u(t_0,\cdot)}_{L^2(B_r\setminus G)}
\notag
&\leq
|B_r\setminus G|^{1/2}
\norm{u(t_0,\cdot)}_{L^\infty(B_r)}
\notag\\
&\leq
C\varepsilon^{1/2}r\,
\norm{\partial_xu(t_0,\cdot)}_{L^2(B_r)}.
\end{align}
Since \(1/(2d)=1/2\) when \(d=1\), the estimates
\eqref{eq:uc_pre_caccioppoli_highd} and
\eqref{eq:uc_pre_caccioppoli_oned} can be written uniformly, for every
\(d\ge1\), as
\begin{equation}\label{eq:uc_pre_caccioppoli}
\norm{u(t_0,\cdot)}_{L^2(B_r)}
\leq
C\varepsilon^{1/(2d)}r\,
\norm{\nabla_xu(t_0,\cdot)}_{L^2(B_r)}.
\end{equation}

To relate the gradient back to \(u\) itself, we use the
vector-valued form of \cite[Equation~(2.20)]{lEscauriaza2006} justified
above by summing the componentwise quadratic estimates and then invoking
\eqref{eq:uc_vector_source}:
\begin{equation}\label{eq:uc_caccioppoli}
\norm{\nabla_xu(t_0,\cdot)}_{L^2(B_r)}
\leq
\frac{C(\theta)}{r}
\norm{u(t_0,\cdot)}_{L^2(B_{2r})},
\end{equation}
where
\begin{equation*}
\theta
=
\frac{\norm{u}_{L^2((t_0,t_0+16)\times B_4)}}{\norm{u(t_0,\cdot)}_{L^2(B_1)}}.
\end{equation*}
Notice that \(\theta\) is finite, since
\(\left[t_0,t_0+16\right]\times\overline{B_4}\) is a compact subset of
\(\left(0,T\right)\times\R^d\) on which \(u\) is continuous, and is well
defined because we are in the case
\(\norm{u\left(t_0,\cdot\right)}_{L^2\left(B_1\right)}>0\). Combining
\eqref{eq:uc_pre_caccioppoli} and \eqref{eq:uc_caccioppoli}, we obtain, for every
\(r<r\left(\varepsilon\right)\wedge r\left(1/2\right)\),
\begin{equation}\label{eq:uc_small_ball}
\norm{u(t_0,\cdot)}_{L^2(B_r)}
\leq
C(\theta)\varepsilon^{1/(2d)}
\norm{u(t_0,\cdot)}_{L^2(B_{2r})}.
\end{equation}
~\\
\noindent\emph{Step 2: Super-polynomial decay and vanishing on one ball.}
Fix \(\varepsilon>0\) sufficiently small that
\begin{equation*}
C(\theta)\varepsilon^{1/(2d)}<1.
\end{equation*}
For each \(r<r\left(\varepsilon\right)\wedge r\left(1/2\right)\), let \(k\left(r\right)\) be the unique
nonnegative integer such that
\begin{equation*}
2^{k(r)}r \leq r(\varepsilon) < 2^{k(r)+1}r.
\end{equation*}
Iterating \eqref{eq:uc_small_ball} exactly \(k\left(r\right)\) times yields
\begin{equation*}
\norm{u(t_0,\cdot)}_{L^2(B_r)}
\leq
\left(C(\theta)\varepsilon^{1/(2d)}\right)^{k(r)}
\norm{u(t_0,\cdot)}_{L^2(B_{2^{k(r)}r})}.
\end{equation*}
Since \(2^{k\left(r\right)}r\le r\left(\varepsilon\right)<1/2\), the ball \(B_{2^{k\left(r\right)}r}\)
is contained in \(B_1\), and therefore
\begin{equation}\label{eq:uc_iterated}
\norm{u(t_0,\cdot)}_{L^2(B_r)}
\leq
\left(C(\theta)\varepsilon^{1/(2d)}\right)^{k(r)}
\norm{u(t_0,\cdot)}_{L^2(B_1)}.
\end{equation}
Moreover, by the definition of \(k\left(r\right)\),
\begin{equation*}
k(r)
\geq
\log_2\left(\frac{r(\varepsilon)}{2r}\right)
=
\log_2\left(\frac{r(\varepsilon)}{r}\right)-1.
\end{equation*}
Hence, \eqref{eq:uc_iterated} implies
\begin{equation}\label{eq:uc_iterated_log}
\norm{u(t_0,\cdot)}_{L^2(B_r)}
\leq
\left(C(\theta)\varepsilon^{1/(2d)}\right)^{\log_2(r(\varepsilon)/r)-1}
\norm{u(t_0,\cdot)}_{L^2(B_1)}.
\end{equation}
We now invoke the vector-valued two-sphere one-cylinder
inequality from \cite[Theorem~2]{lEscauriaza2006}, in the form justified
above. In the present notation it gives
\begin{multline}\label{eq:uc_two_sphere}
\norm{u(t_0,\cdot)}_{L^2(B_1)}
\leq
\norm{u(t_0,\cdot)}_{L^2(B_r)}^{1/(1+C(\theta)\log_2(1/r))}
\\
\times
\left(\sqrt{C(\theta)}\norm{u}_{L^2((t_0,t_0+16)\times B_4)}\right)^{C(\theta)\log_2(1/r)/(1+C(\theta)\log_2(1/r))}.
\end{multline}
Substituting \eqref{eq:uc_iterated_log} into
\eqref{eq:uc_two_sphere}, dividing by the positive quantity
\(\norm{u\left(t_0,\cdot\right)}_{L^2\left(B_1\right)}^{1/\left(1+C\left(\theta\right)\log_2\left(1/r\right)\right)}\),
and raising both sides to the power
\(\left(1+C\left(\theta\right)\log_2\left(1/r\right)\right)/\left(C\left(\theta\right)\log_2\left(1/r\right)\right)\), we
arrive at
\begin{equation*}
\norm{u(t_0,\cdot)}_{L^2(B_1)}
\leq
\left(C(\theta)\varepsilon^{1/(2d)}\right)^{\frac{\log_2(r(\varepsilon)/r)-1}{C(\theta)\log_2(1/r)}}
\sqrt{C(\theta)}
\norm{u}_{L^2((t_0,t_0+16)\times B_4)}.
\end{equation*}
Letting \(r\downarrow 0\), the exponent on the right-hand side
converges to a strictly positive constant depending only on \(\theta\).
Since \(\varepsilon\) can then be sent to \(0\), the right-hand side
converges to \(0\). This contradicts our standing assumption that
\(\norm{u\left(t_0,\cdot\right)}_{L^2\left(B_1\right)}>0\). We conclude that
\begin{equation}\label{eq:uc_ball_one_zero}
\norm{u(t_0,\cdot)}_{L^2(B_1)}=0.
\end{equation}
~\\
\noindent\emph{Step 3: Propagation from one ball to all of \(\R^d\).}
We first record the translated propagation step that will be used
repeatedly. Fix \(y\in\R^d\) and assume that
\begin{equation*}
\norm{u(t_0,\cdot)}_{L^2(B_{1/2}(y))}=0.
\end{equation*}
If
\(\norm{u(t_0,\cdot)}_{L^2(B_1(y))}=0\), there is nothing to prove.
Otherwise, define the parameter associated with the cylinder centered at
\(y\) by
\begin{equation*}
\theta_y
=
\frac{\norm{u}_{L^2((t_0,t_0+16)\times B_4(y))}}
{\norm{u(t_0,\cdot)}_{L^2(B_1(y))}}.
\end{equation*}
This quantity is finite, by the same local boundedness argument as for
\(\theta\) above, and is well defined because its denominator is
positive. Applying the translated vector-valued two-sphere
one-cylinder inequality with center \(y\) and radius \(1/2\) gives
\begin{multline*}
\norm{u(t_0,\cdot)}_{L^2(B_1(y))}
\leq
\norm{u(t_0,\cdot)}_{L^2(B_{1/2}(y))}^{1/(1+C(\theta_y))}
\\
\times
\left(
\sqrt{C(\theta_y)}
\norm{u}_{L^2((t_0,t_0+16)\times B_4(y))}
\right)^{C(\theta_y)/(1+C(\theta_y))}
=0,
\end{multline*}
which contradicts the alternative under consideration. Hence
\begin{equation}\label{eq:uc_propagation_step}
\norm{u(t_0,\cdot)}_{L^2(B_{1/2}(y))}=0
\quad\Longrightarrow\quad
\norm{u(t_0,\cdot)}_{L^2(B_1(y))}=0.
\end{equation}
Now fix \(x_1\in\R^d\) with \(|x_1|\le1/2\). Since
\(B_{1/2}(x_1)\subset B_1\), \eqref{eq:uc_ball_one_zero} and
\eqref{eq:uc_propagation_step} yield
\begin{equation*}
\norm{u(t_0,\cdot)}_{L^2(B_1(x_1))}=0.
\end{equation*}
If \(x_2\in\R^d\) satisfies \(|x_2-x_1|\le1/2\), then
\(B_{1/2}(x_2)\subset B_1(x_1)\), so another application of
\eqref{eq:uc_propagation_step} gives
\begin{equation*}
\norm{u(t_0,\cdot)}_{L^2(B_1(x_2))}=0.
\end{equation*}
Iterating along a finite chain of overlapping unit balls --- for every
\(x\in\R^d\), one can choose centers \(x_1,\dots,x_n=x\) such that
\(\abs{x_1}\le1/2\) and \(\abs{x_{j+1}-x_j}\le1/2\) --- we obtain
\begin{equation*}
\norm{u(t_0,\cdot)}_{L^2(B_1(x))}=0
\qquad\text{for every }x\in\R^d.
\end{equation*}
Thus \(u(t_0,\cdot)=0\) almost everywhere on \(\R^d\). Since
\(u(t_0,\cdot)\) is continuous, it follows that \(u(t_0,x)=0\) for every
\(x\in\R^d\), as claimed.
\end{proof}

\subsection{Backward uniqueness}

The previous lemma is a statement at a single time slice. To make it
useful for the BSDE problem, one needs to propagate the vanishing
forward in time, up to the terminal time. This is the role of
backward uniqueness: once \(u\) is known to vanish identically at some positive
time \(t_0\), the lemma below shows that \(u\) must vanish on the
entire slab \([t_0,T)\times \R^d\).

This is also where the boundedness information from
Section~\ref{sec:extra_reg} becomes essential. The Carleman-based
backward uniqueness argument requires control of both \(u\) and
\(\nabla_xu\) on positive-time slabs, and this is precisely what Corollary~\ref{cor:slab_bounds} provides when applied to \(u=\nabla_x v\).

\begin{lem}[Backward uniqueness]\label{lem:bu}
Let $u \colon \left[0,T\right]\times \R^d \to \R^d$ satisfy
\begin{equation*}u\in C^{1,2}\left(\left(0,T\right)\times \R^d;\R^d\right).
\end{equation*}
Assume \ref{ass:parabolic}, and suppose that there exist constants
$M>0$ and $\kappa>0$ such that
\begin{equation}\label{eq:bu_ass}
\abs{u(t,x)} + \abs{\nabla_xu(t,x)} \leq M,
\qquad
\abs{\Pc u(t,x)} \leq \kappa\left( \abs{u(t,x)} + \abs{\nabla_xu(t,x)} \right),
\end{equation}
for all $\left(t,x\right)\in \left(0,T\right)\times \R^d$. If for some
$t_0\in \left(0,T\right)$ one has $u\left(t_0,\cdot\right)\equiv 0$, then $u \equiv 0$ on $\left[t_0,T\right)\times \R^d$.
\end{lem}

\begin{proof}
We argue by combining a local vanishing result, based on a Carleman estimate, together with a time-iteration argument, following Wu and Zhang
\cite{jWu2019}. The only point that requires clarification is the
passage from scalar-valued to vector-valued functions. In the present
setting, \(\Pc\) acts componentwise on \(u=\left(u^1,\dots,u^d\right)\), and the
relevant Carleman estimate is quadratic in the unknown. Consequently,
although the statements in \cite{jWu2019} are written for scalar
equations, the same Carleman estimate extends to
\(\R^d\)-valued test functions by applying the scalar estimate to each
component and summing. The subsequent local-vanishing argument and its
iteration in time then carry over to the present vector-valued
setting. In particular, no separate pointwise inequalities of the form
$|\Pc u^k|\le C\left(|u^k|+|\nabla_xu^k|\right)$ are needed.
\\

By the affine change of variables
\begin{equation*}
\tau = \frac{t-t_0}{T-t_0},
\qquad
\widetilde u(\tau,x) = u\left(t_0 + (T-t_0)\tau,x\right),
\end{equation*}
it is enough to prove the statement in the normalized case \(t_0=0\)
and \(T=1\). Indeed, \(\widetilde u\) belongs to
\(C^{1,2}\left(\left(0,1\right)\times \R^d;\R^d\right)\), satisfies bounds of the
same type as \eqref{eq:bu_ass}, possibly with different constants, and
the transformed coefficients still satisfy
Assumption~\ref{ass:parabolic}. We therefore work on
\(\left(0,1\right)\times \R^d\) and we continue to denote the rescaled function by $u$ and the rescaled operator by $\mathcal{P}$.
\\~\\
\noindent\emph{Step~1: Local vanishing after the initial time.}
We first show that if \(u\left(0,\cdot\right)=0\), then \(u\) vanishes on a short
time strip \(\left[0,T_1\right]\times \R^d\).
\\

By Assumption~\ref{ass:parabolic} and Remark \ref{rem:forward_roles}, there exist constants
\(0<\lambda\le \Lambda<\infty\) such that
\begin{equation*}
\lambda |\xi|^2 \le \xi^\transpose A(t,x)\xi \le \Lambda |\xi|^2,
\qquad
\xi\in \R^d,
\quad (t,x)\in [0,1]\times \R^d.
\end{equation*}
Moreover, under Assumption~\ref{ass:parabolic}, the coefficients satisfy the hypotheses of \cite[Proposition~1.4]{jWu2019} with the Wu--Zhang growth exponent \(\alpha_{\mathrm{WZ}}=1\). Indeed, the boundedness in \eqref{eq:bu_ass} implies the admissible growth bound
\begin{equation*}
|u(t,x)|\le M e^{M|x|},
\end{equation*}
and, for \(\alpha_{\mathrm{WZ}}=1\), the spatial-decay exponent required of \(\nabla_xA\) in \cite{jWu2019} is zero. Hence the global boundedness of \(\nabla_xA\), together with the ellipticity and time regularity recorded in Remark~\ref{rem:forward_roles}, is sufficient. Thus there exists a constant \(K>0\), depending only on the dimension, the ellipticity constants, and the coefficient bounds entering \cite[Proposition~1.4]{jWu2019}, such that, with
\begin{equation*}
b=(8\Lambda)^{-1},\qquad \langle x\rangle=(1+|x|^2)^{1/2},
\end{equation*}
for every scalar-valued test function
\(\phi\in C_0^\infty\left(\left(0,1\right)\times \R^d\right)\) and every
\(\gamma>0\),
\begin{multline}\label{eq:bu_carleman}
\int_{(0,1)\times \R^d}
e^{2\gamma (t^{-K}-1)-\frac{b\langle x\rangle+K}{t}}
\left( |\phi(t,x)|^2 + |\nabla_x \phi(t,x)|^2 \right)\Id x\Id t
\\
\le
\int_{(0,1)\times \R^d}
e^{2\gamma (t^{-K}-1)-\frac{b\langle x\rangle+K}{t}}
|\Pc \phi(t,x)|^2\Id x\Id t.
\end{multline}
We claim that the same estimate holds for
\(\R^d\)-valued test functions. Indeed, if 
\begin{equation*}
\varphi=(\varphi^1,\dots,\varphi^d)\in
C_0^\infty\left((0,1)\times \R^d;\R^d\right),
\end{equation*}
then applying
\eqref{eq:bu_carleman} to each component \(\varphi^k\) and summing
over \(k=1,\dots,d\) gives
\begin{multline*}
\int_{(0,1)\times \R^d}
e^{2\gamma (t^{-K}-1)-\frac{b\langle x\rangle+K}{t}}
\sum_{k=1}^d \left(|\varphi^k(t,x)|^2 + |\nabla_x\varphi^k(t,x)|^2\right)
\Id x\Id t
\\
\le
\int_{(0,1)\times \R^d}
e^{2\gamma (t^{-K}-1)-\frac{b\langle x\rangle+K}{t}}
\sum_{k=1}^d |\Pc \varphi^k(t,x)|^2\Id x\Id t.
\end{multline*}
Since
\begin{equation*}
\sum_{k=1}^d |\varphi^k|^2 = |\varphi|^2,
\quad
\sum_{k=1}^d |\nabla_x\varphi^k|^2 = |\nabla_x\varphi|^2\quad\mbox{and}\quad
\sum_{k=1}^d |\Pc\varphi^k|^2 = |\Pc\varphi|^2,
\end{equation*}
this is exactly the corresponding Carleman estimate for vector-valued
test functions.
\\

We now adapt the proof of \cite[Lemma~2.1]{jWu2019}. In the scalar
case, that lemma shows that a function satisfying a differential
inequality of the form
\begin{equation*}
|\Pc \phi|\le \widetilde\kappa\left(|\phi|+|\nabla_x\phi|\right),
\end{equation*}
together with a suitable growth condition, and vanishing at time \(0\), vanishes on a strip \(\left(0,T_1\right)\times \R^d\) for some \(T_1\in \left(0,1\right)\).
In our setting, the boundedness in \eqref{eq:bu_ass} is
stronger than the \(\alpha_{\mathrm{WZ}}=1\) growth condition used there. The spatial factor
\(e^{-b\langle x\rangle/t}\) in \eqref{eq:bu_carleman} is precisely the decay needed to pass to the limit in the spatial cutoffs in the proof of \cite[Lemma~2.1]{jWu2019}. Since that proof relies only on this Carleman estimate, the differential inequality, and the corresponding growth control, the same argument applies to the \(\R^d\)-valued function \(u\), using the vector-valued Carleman estimate established above. We therefore obtain the existence of \(T_1\in \left(0,1\right)\) such that
\begin{equation*}
u \equiv 0
\qquad \text{on } (0,T_1)\times \R^d.
\end{equation*}
Together with \(u\left(0,\cdot\right)=0\) and the continuity of \(u\), this gives
\begin{equation*}
u \equiv 0
\qquad \text{on } [0,T_1]\times \R^d.
\end{equation*}
~\\
\noindent\emph{Step~2: Propagation up to the terminal time.}
We now iterate the local vanishing result in time, exactly as in the
proof of \cite[Theorem~1.2]{jWu2019}. Define
\begin{equation*}
S
=
\left\{ s\in [0,1] : u \equiv 0 \text{ on } [0,s]\times \R^d \right\}.
\end{equation*}
By Step~1, the set \(S\) is nonempty. Let \(s_*=\sup S\). We claim
that \(s_*=1\). Suppose for contradiction that \(s_*<1\). Since \(u\)
is continuous and vanishes on \(\left[0,s\right]\times \R^d\) for every
\(s<s_*\), we have
\begin{equation*}
u(s_*,\cdot)=0 \qquad \text{on } \R^d.
\end{equation*}
Consider the shifted and rescaled function
\begin{equation*}
\widehat u(\tau,x)=u\left(s_*+(1-s_*)\tau,x\right),
\qquad
(\tau,x)\in [0,1]\times \R^d.
\end{equation*}
Exactly as before, \(\widehat u\) satisfies bounds of the form
\eqref{eq:bu_ass} on \(\left(0,1\right)\times \R^d\), possibly with different
constants, and \(\widehat u\left(0,\cdot\right)=0\). Applying Step~1 to
\(\widehat u\), we find \(\delta\in \left(0,1\right)\) such that
\begin{equation*}
\widehat u \equiv 0
\qquad \text{on } [0,\delta]\times \R^d.
\end{equation*}
Returning to the original time variable, this means that
\begin{equation*}
u \equiv 0
\qquad \text{on } [0,s_*+\delta(1-s_*)]\times \R^d,
\end{equation*}
which contradicts the definition of \(s_*\) as the supremum of \(S\).
Hence \(s_*=1\), and therefore
\begin{equation*}
u \equiv 0 \qquad \text{on } [0,1)\times \R^d.
\end{equation*}
Undoing the initial affine change of variables yields
\begin{equation*}
u \equiv 0 \qquad \text{on } [t_0,T)\times \R^d,
\end{equation*}
which proves the lemma.
\end{proof}

\subsection{Zero-set rigidity and the contradiction argument}

We now combine the two preceding lemmas to derive the rigidity
statement that forms the decisive PDE input for the density theorem.
It asserts that, at every positive time \(t\), the set of points where
\(\nabla_x v\left(t,\cdot\right)\) vanishes cannot have positive Lebesgue
measure. Combined with the absolute continuity of the law
of \(X_t\) this deterministic statement on the zero set of \(\nabla_x v\left(t,\cdot\right)\) transfers to an almost sure non-vanishing statement for \(\nabla_x v\left(t,X_t\right)\).

The argument proceeds by contradiction. If for some
\(t_0\in\left(0,T\right)\) the zero set of \(u\left(t_0,\cdot\right)\) had positive
measure, unique continuation would imply that \(u\left(t_0,\cdot\right)\)
vanishes identically. Backward uniqueness would then force
\(u\equiv 0\) on \(\left[t_0,T\right)\times\R^d\). Consequently,
\(v\left(t,\cdot\right)\) would be spatially constant for all \(t\in\left[t_0,T\right)\),
and passage to the limit as \(t\uparrow T\) would show that the
terminal condition \(g\) is constant, in contradiction with the assumed non-constancy of \(g\). The proposition below makes
this reasoning precise.

\begin{prop}[Zero sets of \(\nabla_x v\)]\label{prop:mm_inv}
Assume \ref{ass:rigidity}, and suppose that \(g\) is non-constant. Let \((v,w)\) be
the Markovian solution fixed in Assumption~\ref{ass:backward}. Then for
every \(t\in(0,T)\), the set
\begin{equation*}
\Uc_t  =  \left\{ x\in \R^d : \nabla_x v(t,x) = 0 \right\}
\end{equation*}
has Lebesgue measure zero. In particular, $\nabla_x v\left(t,X_t^{0,x_0}\right)
\neq 0$ almost surely for every $t\in \left(0,T\right)$.
\end{prop}

\begin{proof}[Proof of Proposition \ref{prop:mm_inv}]
Fix \(t_0\in \left(0,T\right)\) and set \(u = \nabla_x v\). By
Proposition~\ref{prop:extra_reg}, the field \(u\) belongs to
\(C^{1,2}_{\loc}\) and satisfies \eqref{eq:struct_ineq}. Suppose, for
contradiction, that the zero set
\begin{equation*}
\Uc_{t_0} = \{x\in \R^d : u(t_0,x)=0\}
\end{equation*}
has positive Lebesgue measure. By Lemma~\ref{lem:uc}, this implies that \(u\left(t_0,\cdot\right)\equiv 0\) on \(\R^d\). We now want to apply backward uniqueness. Lemma~\ref{lem:bu}
requires the bounds \eqref{eq:bu_ass} on its whole time interval and a
vanishing time interior to that interval. The boundedness is available,
through Corollary~\ref{cor:slab_bounds}, on slabs staying strictly away
from \(0\) and \(T\). We therefore apply the lemma on such a slab
containing \(t_0\) in its interior. Fix an arbitrary
\(\varepsilon\in \left(0,\min\{t_0,T-t_0\}\right)\) and apply
Lemma~\ref{lem:bu} with the interval
\(\left[\varepsilon,T-\varepsilon\right]\) in the role of
\(\left[0,T\right]\) and with vanishing time
\(t_0\in \left(\varepsilon,T-\varepsilon\right)\). Its hypotheses hold
there: \(u\) belongs to
\(C^{1,2}\left(\left(\varepsilon,T-\varepsilon\right)\times \R^d;\R^d\right)\)
and satisfies \eqref{eq:struct_ineq} by
Proposition~\ref{prop:extra_reg}, as noted above, and, since
\(\varepsilon<\min\{t_0,T-t_0\}\le T/2\),
Corollary~\ref{cor:slab_bounds} yields
\begin{equation*}
\sup_{(t,x)\in [\varepsilon,T-\varepsilon]\times \R^d}
\left( \abs{u(t,x)} + \abs{\nabla_xu(t,x)} \right) < \infty.
\end{equation*}
Therefore, Lemma~\ref{lem:bu} gives
\begin{equation*}
u \equiv 0 \quad \text{on } [t_0,T-\varepsilon)\times \R^d.
\end{equation*}
As \(\varepsilon\in \left(0,\min\{t_0,T-t_0\}\right)\) was arbitrary, we
conclude that
\begin{equation*}
u \equiv 0 \quad \text{on } [t_0,T)\times \R^d.
\end{equation*}

At this point the contradiction is immediate. Since \(u=\nabla_x v\),
the function \(x\mapsto v\left(t,x\right)\) is constant on \(\R^d\) for each
\(t\in \left[t_0,T\right)\). Let us denote this constant by \(c\left(t\right)\). By
Proposition~\ref{prop:reg}, the function \(v\) is continuous on
\(\left[0,T\right]\times \R^d\), so letting \(t\uparrow T\) gives, for every
\(x\in \R^d\),
\begin{equation*}
g(x) = v(T,x) = \lim_{t\uparrow T} v(t,x) = \lim_{t\uparrow T} c(t).
\end{equation*}
Thus, \(g\) is constant on \(\R^d\), contradicting the assumed non-constancy of \(g\). Hence, the set \(\Uc_{t_0}\)
must have Lebesgue measure zero.
\\

Finally, Assumption \ref{ass:parabolic} implies that
the uniformly elliptic diffusion \(X_{t_0}^{0,x_0}\) admits a density
with respect to the Lebesgue measure -- see Aronson \cite{Aronson}. Since
\(|\Uc_{t_0}|=0\), it follows that
\begin{equation*}
\mathbb{P}\left( \nabla_x v(t_0,X_{t_0}^{0,x_0}) = 0 \right) = 0.
\end{equation*}
As \(t_0\in \left(0,T\right)\) was arbitrary, the proof is complete.
\end{proof}

\section{Proof of the main density theorem}\label{sec:main_proof}

We now complete the proof of the main result. Proposition~\ref{prop:mm_inv}
shows that, at every fixed positive time, the spatial gradient of the
decoupling field does not vanish at the random point \(X_t^{0,x_0}\) almost
surely. Lemma~\ref{lem:malliavin_rep} expresses the Malliavin derivative of
\(Y_t^{0,x_0}\) in terms of the same gradient. Combining these two facts with
the Bouleau--Hirsch criterion converts the rigidity result into absolute
continuity of the law of \(Y_t^{0,x_0}\), exactly as outlined in the
introduction.

\begin{proof}[Proof of Theorem \ref{thm:main}]
Fix \(t\in \left(0,T\right)\). By Proposition~\ref{prop:mm_inv},
\begin{equation}\label{eq:grad_nonzero}
\nabla_x v(t,X_t^{0,x_0}) \neq 0
\quad \text{a.s.}
\end{equation}
This is the decisive deterministic input produced by the rigidity
argument of the previous section.
We now turn to the Malliavin derivative of \(Y_t^{0,x_0}\). By
Lemma~\ref{lem:malliavin_rep}, the random variable \(Y_t^{0,x_0}\)
belongs to \(\Dcal^{1,2}\), and for every \(r\in \left[0,T\right]\),
\begin{equation*}
D_rY_t^{0,x_0}
=
\nabla_x v(t,X_t^{0,x_0})D_rX_t^{0,x_0}\mathbf 1_{[0,t]}(r).
\end{equation*}
Thus, the Malliavin derivative of \(Y_t^{0,x_0}\) factorizes through
the spatial gradient of the decoupling field evaluated along the
forward diffusion. To prove that \(Y_t^{0,x_0}\) has a density, it suffices to show that this derivative is almost surely
nonzero in \(L^2\left(\left[0,T\right]\right)\).
\\

For \(0\le r\le t\), the standard variational representation of the Malliavin derivative gives
\begin{equation*}
D_rX_t^{0,x_0}
=
J_{r,t}\sigma(r,X_r^{0,x_0})
\qquad \text{for a.e. }r\in[0,t].
\end{equation*}
By the flow property of the Jacobian and Lemma~\ref{lem:jacobian_flow},
\(J_{r,t}=J_{0,t}J_{0,r}^{-1}\), and hence
\begin{equation}\label{eq:DX_flow_rep}
D_rX_t^{0,x_0}
=
J_{0,t}J_{0,r}^{-1}\sigma(r,X_r^{0,x_0})
\qquad \text{for a.e. }r\in[0,t],
\end{equation}
where the Jacobians are evaluated along the trajectory starting from
\(x_0\). Define, for \(r\in[0,t]\),
\begin{equation*}
H_r
:=
J_{0,t}J_{0,r}^{-1}\sigma(r,X_r^{0,x_0}).
\end{equation*}
By Lemma~\ref{lem:jacobian_flow}, the processes \(r\mapsto J_{0,r}\) and
\(r\mapsto J_{0,r}^{-1}\) have continuous paths, and the continuity of
\(X\) together with Assumptions~\ref{ass:flow} and \ref{ass:parabolic}
shows that \(r\mapsto H_r\) is continuous. Moreover,
\begin{equation*}
H_t=\sigma(t,X_t^{0,x_0}).
\end{equation*}
Hence, by the uniform ellipticity assumption \ref{it:P3},
\begin{equation*}
H_tH_t^\transpose
=
\sigma(t,X_t^{0,x_0})\sigma(t,X_t^{0,x_0})^\transpose
\ge \lambda I_d.
\end{equation*}
Consequently, for almost every sample point there exists
\(\delta=\delta(\omega)\in(0,t]\) such that
\begin{equation*}
H_rH_r^\transpose\ge \frac{\lambda}{2}I_d,
\qquad r\in[t-\delta,t].
\end{equation*}
Using \eqref{eq:DX_flow_rep}, the Malliavin covariance matrix of
\(X_t^{0,x_0}\) can therefore be written as
\begin{equation*}
C_t
=
\int_0^t H_rH_r^\transpose\,\Id r,
\end{equation*}
and satisfies, pathwise,
\begin{equation}\label{eq:Ct_positive}
C_t
\ge
\frac{\lambda\delta}{2}I_d.
\end{equation}
Thus \(C_t\) is positive definite almost surely. In particular,
\begin{equation*}
\norm{DY_t^{0,x_0}}_{L^2([0,T])}^2
=
\nabla_x v(t,X_t^{0,x_0})
C_t
\nabla_x v(t,X_t^{0,x_0})^\transpose.
\end{equation*}
Combining \eqref{eq:Ct_positive} with \eqref{eq:grad_nonzero}, we obtain
\begin{equation*}
\norm{DY_t^{0,x_0}}_{L^2([0,T])} > 0
\quad \text{a.s.}
\end{equation*}
We are now exactly in the setting of the Bouleau--Hirsch criterion.
Since \(Y_t^{0,x_0}\in \Dcal^{1,2}\) and its Malliavin derivative is
almost surely nonzero in \(L^2\left(\left[0,T\right]\right)\), the law of \(Y_t^{0,x_0}\) is absolutely continuous with respect to the
Lebesgue measure on \(\R\) -- see Nualart
\cite[Theorem~2.1.2]{dNualart}. This completes the proof.
\end{proof}

\section{Applications}\label{sec:applications}

We close with two examples whose role is not to develop new
application-specific theory, but rather to show concretely how the
abstract density result applies in genuinely multidimensional
Markovian models arising in different fields. In both examples, the
state process $X_t \in \R^d$, $d\geq 2$, collects several interacting factors, while the scalar backward component
\begin{equation*}
Y_t = v\left( t,X_t \right)
\end{equation*}
represents the quantity of interest at time \(t\). Our theorem gives a structural conclusion on the law of \(Y_t\): at each strictly
positive time, this law is absolutely continuous with respect to the Lebesgue measure.

The point of the section is therefore twofold. First, it illustrates that the assumptions of Theorem~\ref{thm:main} are compatible with models that are not merely one-dimensional toy examples. Second, it shows how the density conclusion can be interpreted in context: in applications, the absence of atoms in the law of \(Y_t\) is relevant whenever one studies threshold events, quantiles, calibration targets, or uncertainty summaries derived from the intermediate-time distribution of the response $Y_t$.

\subsection{Multifactor incomplete-market valuation with volatility, rates, and credit factors}

Incomplete-market valuation is a natural source of genuinely multidimensional Markovian BSDEs. The value of a contingent claim then depends not only on tradable assets but also on nontradable risk factors, and, because the exposure to these factors cannot be replicated by trading in the available assets, the value process is characterized by a backward equation rather than by the complete-market pricing formula: discounting, running costs, and risk premia for the nontradable factors enter through the generator. See El Karoui, Peng, and Quenez \cite{KPQ} for the general BSDE viewpoint, Henderson and Liang \cite{HendersonLiang2016} for a multidimensional incomplete-market model with counterparty-risk effects, and Schwarz \cite{SchwarzMarketCompletion} for a complementary approach that restores completeness at the outset by enlarging the set of traded assets with derivatives written on the nontradable factors. We work directly with the incomplete-market BSDE and choose the simplest valuation rule that exercises the hypotheses of Theorem~\ref{thm:main}: the state dimension is arbitrary, and the state enters the generator only through a Girsanov-type kernel -- in financial terms, a market price of risk -- multiplying the martingale integrand.
\\

Concretely, let
\begin{equation*}
X_t=\left( L_t,V_t,R_t,\Lambda_t \right) \in \R^{m+k+\ell+q},
\qquad d=m+k+\ell+q \geq 2,
\end{equation*}
where $L_t=\left( L_t^1,\dots,L_t^m \right)$ denotes tradable
log-prices, $V_t \in \R^k$ collects stochastic-volatility or
convenience-yield factors, $R_t \in \R^\ell$ collects short-rate or
funding variables, and $\Lambda_t \in \R^q$ contains credit, liquidity, or counterparty-risk factors. We assume that under a reference pricing measure the full factor process solves
\begin{equation*}
\Id X_t=b\left( t,X_t \right)\Id t+\sigma\left( t,X_t \right)\Id W_t,
\qquad X_0=x_0,
\end{equation*}
with coefficients satisfying Assumption \ref{ass:parabolic}.
\\

The time-$t$ value of the claim and its sensitivity to the driving Brownian motion are the scalar backward component $Y_t$ and the martingale integrand $Z_t$ of the linear Markovian BSDE
\begin{equation*}
Y_t
=
g\left( X_T \right)
+
\int_t^T
\left(
\rho\left( s \right)Y_s
+
c\left( s \right)
+
\Theta\left( s,X_s \right)\cdot Z_s
\right)\Id s
-
\int_t^T Z_s\Id W_s,
\end{equation*}
where $\rho \colon \left[0,T\right]\to\R$ and $c \colon \left[0,T\right]\to\R$ are bounded and $\alpha/2$-H\"older continuous deterministic functions, and $\Theta \colon \left[0,T\right]\times\R^d\to\R^d$ is bounded and continuously differentiable in $x$ with bounded first spatial derivatives, both $\Theta$ and $\nabla_x\Theta$ being $\alpha/2$-H\"older regular in $t$ and $\alpha$-H\"older regular in $x$. The term $\rho\left( s \right)Y_s$ encodes discounting, $c\left( s \right)$ a running cash flow or cost, and $\Theta\left( s,X_s \right)\cdot Z_s$ a factor-dependent risk premium, with $\Theta$ playing the role of a state-dependent market price of risk. The vector $\Theta$ is the sole source of state dependence in the generator, and it enters multiplied by $Z_s$: the structural condition \eqref{eq:struct_fx_zero} therefore holds by construction, and the compatibility requirement of Theorem~\ref{thm:main} is satisfied automatically by this class of models rather than imposed on it.
\\

For the terminal condition, one may take a smooth bounded profile of the log-price and factor variables. A representative example is
\begin{equation*}
g\left( x \right)
=
\Phi\left(
\sum_{i=1}^m \omega_i x_i,
\ x_1-x_2,
\ x_{m+1},\dots,x_d
\right),
\end{equation*}
where $m\ge 2$, $\omega_1,\dots,\omega_m>0$ are normalized basket weights satisfying
\begin{equation*}
\sum_{i=1}^m\omega_i=1,
\end{equation*}
and $\Phi \in C_b^\infty\left( \R^{d-m+2} \right)$ is non-constant and has bounded first derivatives. The first argument is therefore the logarithm of the weighted geometric mean $\prod_{i=1}^m \left( e^{x_i} \right)^{\omega_i}$ of the tradable prices (a geometric basket), and the second is the log-ratio $x_1-x_2=\log\left( e^{x_1}/e^{x_2} \right)$, recording the relative performance of the first two assets. Writing $g=\Phi\circ \Pi$, where $\Pi\colon \R^d\to\R^{d-m+2}$ denotes the linear map given by the argument of $\Phi$ above, one has $\nabla_x g\left( x \right)=\Pi^{\transpose}\nabla\Phi\left( \Pi x \right)$, so that $g\in C_b^1\left( \R^d \right)$ with $\norm{\nabla_x g}_\infty\le \norm{\Pi}\norm{\nabla\Phi}_\infty<\infty$. Moreover, since the weights are positive, the linear functionals $x\mapsto \sum_{i=1}^m \omega_i x_i$ and $x\mapsto x_1-x_2$ are linearly independent, so $\Pi$ is surjective and $g$ is non-constant whenever $\Phi$ is. In words, the terminal value may depend simultaneously on the prices of several tradable assets and on additional nontradable factors, which is precisely the regime in which a multidimensional state description is natural.
\\

The next proposition shows that the abstract density theorem applies directly to this valuation problem.
\begin{prop}[Density for a multifactor valuation process]\label{prop:example_finance}
Assume, in addition to the setup above, that $g \in C_b^1\left( \R^d
\right)$ is non-constant with bounded gradient, and that conditions~\ref{it:R1}--\ref{it:R2} of Assumption~\ref{ass:extra_reg} hold for the coefficients
\(\left( b,\sigma,f \right)\), where
\begin{equation*}
f\left( t,x,y,z \right)=\rho\left( t \right)y+c\left( t \right)+\Theta\left( t,x \right)\cdot z.
\end{equation*}
Then, all assumptions of Theorem~\ref{thm:main} are
satisfied. Consequently, for every \(t\in\left(0,T\right)\), the time-\(t\)
valuation variable \(Y_t\) admits a density with respect to the Lebesgue measure on \(\R\).
\end{prop}

\begin{proof}
The assumptions on the forward diffusion are part of the setup, so it
remains to verify the structural conditions on the generator. We have
$\partial_y f\left( t,x,y,z \right)=\rho\left( t \right)$, $\nabla_z
f\left( t,x,y,z \right)=\Theta\left( t,x \right)$ and
\begin{equation*}
\nabla_x f\left( t,x,y,z \right)
=
\left( \nabla_x\Theta\left( t,x \right) \right)^{\transpose} z.
\end{equation*}
Since \(\rho\) is bounded, \(\partial_y f\) is globally bounded, and
since \(\Theta\) is bounded, so is \(\nabla_z f\). Since neither depends on
\(\left(y,z\right)\), the Lipschitz condition of
Assumption~\ref{ass:rigidity}\ref{it:U2} holds trivially, and
\ref{it:U2} is verified. Moreover,
\begin{equation*}
\nabla_x f\left( t,x,y,0 \right)=0,
\end{equation*}
and the map \(z \mapsto \nabla_x f\left( t,x,y,z \right)\) is globally
Lipschitz, uniformly in \(\left(t,x,y\right)\), with Lipschitz constant
controlled by \(\norm{\nabla_x\Theta}_\infty\). This verifies
Assumption~\ref{ass:rigidity}\ref{it:U3}. A fortiori,
\eqref{eq:lip_fyz} and Assumption~\ref{ass:backward}\ref{it:B3} hold on
\(\Sc\), by \ref{it:U2} and \ref{it:U3} respectively. The H\"older regularity required in
Assumption~\ref{ass:backward}\ref{it:B2} is local in
\(\left(y,z\right)\), and therefore follows from the assumed H\"older
regularity of \(\rho\), \(c\), and \(\Theta\). Thus, the generator
assumptions are satisfied.

By hypothesis, \ref{it:R1}--\ref{it:R2} hold, and the terminal condition \(g\) is non-constant with bounded gradient. Moreover, since
\(\rho\), \(c\), and \(\Theta\) are bounded, the generator \(f\) is globally
Lipschitz in \((y,z)\) and of linear growth, and \(g\in C_b^1(\R^d)\). Hence
a Markovian solution exists by the classical results recalled in
Remark~\ref{rem:CYCZ}, and the existence clause of
Assumption~\ref{ass:backward} is satisfied. Therefore, all assumptions of
Theorem~\ref{thm:main} are fulfilled, which yields the existence of a
density for \(Y_t\) for every \(t\in(0,T)\).
\end{proof}

\begin{rem}
Two comments are in order. First, the generator is linear in $\left( y,z \right)$, so the valuation admits the classical explicit representation as a discounted conditional expectation under the Girsanov change of measure with kernel $\Theta$ (see \cite{KPQ}). Even so, for state dimension $d\ge 2$ the absolute continuity of the law of $Y_t$ is not immediate from this representation, nor is it covered by the one-dimensional criteria of \cite{AB,AKH,MPR}. It hinges on the almost-everywhere non-vanishing of $\nabla_x v\left( t,\cdot \right)$, which is precisely what the rigidity mechanism behind Theorem~\ref{thm:main} supplies. Second, from the financial point of view, the conclusion says that the marginal law of the value of the claim is absolutely continuous at each fixed time. In particular, it has no atoms. This is the relevant regularity whenever one studies quantiles or tail probabilities of the value process -- for instance, in value-at-risk or expected-shortfall calculations.
\end{rem}

The previous example comes from finance. The next one, from cardiac electrophysiology, illustrates that the theorem is not tied to a particular application area. What matters is the structural form of the Markovian BSDE and the regularity of its coefficients.

\subsection{A viscous-eikonal model for cardiac activation times}

The electrical impulse that triggers a heartbeat spreads through the heart muscle as a traveling wave, and simplified models describe the propagation through the wave's arrival time at each point -- the activation time. The activation time satisfies an eikonal equation, which prescribes the norm of its spatial gradient in terms of the local propagation speed. In cardiac tissue this speed is direction-dependent, conduction being faster along muscle fibers than across them, and a second-order term is often added as a regularization -- the viscous version of the model. Such models are used to reconstruct activation times from measurements, see Kunisch and Trautmann \cite{KunischTrautmann2022}. The eikonal--curvature description itself goes back to Keener \cite{Keener1991}. In contrast with the previous example, these models lead to a generator that is genuinely nonlinear in the gradient variable.
\\

To place this in our framework, let the state $X_t\in\R^d$, $d\geq 2$, solve
\begin{equation*}
\Id X_t=b\left( t,X_t \right)\Id t+\sigma\left( t,X_t \right)\Id W_t,
\qquad X_0=x_0,
\end{equation*}
where \(b\) and \(\sigma\) satisfy Assumption \ref{ass:parabolic}. The components of \(X_t\) collect the quantities on which the activation time depends -- a location in the tissue and tissue parameters such as the local propagation speed. The diffusion models the uncertainty in these quantities. Let $M\colon\left[0,T\right]\times\R^{d}\to\R^{d\times d}$ be a smooth, bounded, symmetric, uniformly positive definite matrix field, encoding the direction dependence of the propagation speed, let \(\beta>0\) be a constant, and let $\rho \colon \left[0,T\right]\to\R$ be bounded and $\alpha/2$-H\"older continuous. We consider the parabolic viscous-eikonal equation
\begin{equation*}
\left\{
\begin{aligned}
&\partial_t v+\Lc v+\rho(t)v
+c(t,x)
\left(
\sqrt{\beta+(\nabla_x v)M(t,x)(\nabla_x v)^{\transpose}}
-\sqrt{\beta}
\right)=0,\\
&v(T,x)=g(x),
\end{aligned}
\right.
\end{equation*}
where \(\Lc\) denotes the generator of the forward diffusion, defined in Remark~\ref{rem:struct_fx_necessary}. The function $c\colon\left[0,T\right]\times\R^d\to\left(0,\infty\right)$ is bounded and continuously differentiable in \(x\), with bounded first spatial derivatives. Both \(c\) and \(\nabla_x c\) are $\alpha/2$-H\"older regular in \(t\) and $\alpha$-H\"older regular in \(x\). The solution \(v\) is the activation-time variable of the regularized model: the square-root term is a smoothed, direction-weighted norm of the gradient, the constant \(\beta>0\) replacing the nonsmooth map \(p \mapsto \abs{p}\) by the smooth approximation
\begin{equation*}
p \mapsto \sqrt{\beta+\abs{p}^2}-\sqrt{\beta},
\end{equation*}
which removes the corner at the origin while preserving the growth for large gradients. In BSDE form, this leads to
\begin{equation*}
Y_t
=
g\left( X_T \right)
+
\int_t^T f\left( s,X_s,Y_s,Z_s \right)\Id s
-
\int_t^T Z_s\Id W_s,
\end{equation*}
with generator

\begin{equation*}
f\left( t,x,y,z \right)
=
\rho\left( t \right)y
+
c\left( t,x \right)
\left(
\sqrt{\beta+\abs{zB\left( t,x \right)^{\transpose}}^2}
-
\sqrt{\beta}
\right).
\end{equation*}
Here, $B\colon\left[0,T\right]\times\R^d\to\R^{d\times d}$ is bounded, uniformly non-degenerate, and continuously differentiable in \(x\) with bounded first spatial derivatives, both \(B\) and \(\nabla_x B\) being $\alpha/2$-H\"older regular in \(t\) and $\alpha$-H\"older regular in \(x\). Uniform non-degeneracy means that there exists \(\lambda_B>0\) such that

\begin{equation*}
\abs{zB\left( t,x \right)^{\transpose}}\ge \lambda_B\abs{z},
\qquad
(t,x,z)\in [0,T]\times\R^d\times\R^d.
\end{equation*}
The BSDE matches the viscous-eikonal equation exactly and Proposition~\ref{prop:reg}(b) gives \(Z_t=\left( \left( \nabla_x v \right)\sigma \right)\left( t,X_t \right)\). Let \(M^{1/2}\) denote the symmetric positive-definite square root of \(M\) and choose \(B=M^{1/2}\sigma^{-\transpose}\). Then, at \(z=\left( \nabla_x v \right)\sigma\),
\begin{equation*}
zB^{\transpose}
=
\left(\nabla_x v\right)\sigma\,\sigma^{-1}M^{1/2}
=
\left(\nabla_x v\right)M^{1/2},
\end{equation*}
and therefore
\begin{equation*}
\abs{zB^{\transpose}}^2
=
\left(\nabla_x v\right)M\left(\nabla_x v\right)^{\transpose}.
\end{equation*}
Thus the equation for \(v\) is precisely the semilinear equation \eqref{eq:PDE_v} associated with the BSDE. For this choice, the boundedness and uniform non-degeneracy of \(B\) follow from \ref{it:P1}, \ref{it:P3}, and the boundedness and uniform positive definiteness of \(M\), while the regularity imposed on \(B\) becomes a hypothesis on \(M\) and \(\sigma\). Thus \(Y_t=v\left( t,X_t \right)\) is the activation-time variable evaluated along the random state, and \(Z_t\) its sensitivity to the driving Brownian motion. As in the previous example, the state dependence of the generator -- through \(c\) and \(B\) -- is confined to a term that vanishes identically at \(z=0\). The structural condition \eqref{eq:struct_fx_zero} therefore holds by construction.
\\

The following proposition shows that the main theorem applies to this model as well.

\begin{prop}[Density for a viscous-eikonal activation variable]\label{prop:example_cardiac}
Under the setup above, assume that $g \in C_b^1\left( \R^d \right)$ is
non-constant with bounded gradient and that conditions~\ref{it:R1}--\ref{it:R2} of Assumption~\ref{ass:extra_reg} hold for the coefficients holds for the coefficients
\(\left( b,\sigma,f \right)\). Then, all
assumptions of Theorem~\ref{thm:main} are satisfied. Consequently, for
every \(t\in\left(0,T\right)\), the activation variable \(Y_t\) admits a density
with respect to the Lebesgue measure on \(\R\).
\end{prop}

\begin{proof}
The assumptions on the forward diffusion are part of the setup, so it
remains to verify the conditions on the generator and on the terminal
datum. The generator is continuously differentiable in
\(\left( x,y,z \right)\), with
\(\partial_y f\left( t,x,y,z \right)=\rho\left( t \right)\),

\begin{equation*}
\nabla_z f\left( t,x,y,z \right)
=
c\left( t,x \right)
\frac{zB\left( t,x \right)^{\transpose}B\left( t,x \right)}
{\sqrt{\beta+\abs{zB\left( t,x \right)^{\transpose}}^2}},
\end{equation*}
and, for \(k=1,\dots,d\), suppressing the \(\left( t,x \right)\)-arguments
of \(c\) and \(B\),

\begin{equation*}
\partial_{x_k} f\left( t,x,y,z \right)
=
\left( \partial_{x_k}c \right)
\left(
\sqrt{\beta+\abs{zB^{\transpose}}^2}-\sqrt{\beta}
\right)
+
c\,
\frac{\left\langle zB^{\transpose},z\left( \partial_{x_k}B \right)^{\transpose}\right\rangle}
{\sqrt{\beta+\abs{zB^{\transpose}}^2}}.
\end{equation*}
We first verify \ref{it:U2}. Neither \(\partial_y f\) nor
\(\nabla_z f\) depends on \(y\), and \(\partial_y f\) does not depend on
\(z\). The map \(q\mapsto q/\sqrt{\beta+\abs{q}^2}\) is bounded by one
and, since \(\beta>0\), globally Lipschitz. Since
\begin{equation*}
\nabla_z f(t,x,y,z)
=
c(t,x)\,
\frac{zB(t,x)^{\transpose}}{\sqrt{\beta+\abs{zB(t,x)^{\transpose}}^2}}
B(t,x),
\end{equation*}
and \(\rho\), \(c\), and \(B\) are bounded, \(\partial_y f\) and \(\nabla_z f\) are therefore globally bounded, and \(z\mapsto \nabla_z f\left( t,x,y,z \right)\) is Lipschitz uniformly in \(\left( t,x \right)\). This verifies
\ref{it:U2}.

We next verify \ref{it:U3}. Both terms of \(\partial_{x_k}f\) vanish at
\(z=0\), giving the structural identity
\(\nabla_x f\left( t,x,y,0 \right)=0\). Their \(z\)-dependent factors
have uniformly bounded \(z\)-derivatives -- the first by
\(\norm{B}_\infty\), the second by the boundedness of \(B\) and
\(\nabla_x B\) together with the lower bound
\(\abs{zB^{\transpose}}\ge\lambda_B\abs{z}\) -- and the coefficients
\(\partial_{x_k}c\) and \(c\) are bounded, so
\(z\mapsto \nabla_x f\left( t,x,y,z \right)\) is globally Lipschitz,
uniformly in \(\left( t,x,y \right)\). This verifies \ref{it:U3}.

A fortiori, \eqref{eq:lip_fyz} and
Assumption~\ref{ass:backward}\ref{it:B3} hold on \(\Sc\), by
\ref{it:U2} and \ref{it:U3} respectively, and the H\"older condition
\eqref{eq:holder_f_tx} of \ref{it:B2} holds locally uniformly in
\(\left( y,z \right)\), by the H\"older regularity of \(\rho\), \(c\),
and \(B\) and the Lipschitz continuity of
\(q\mapsto\sqrt{\beta+\abs{q}^2}\). By hypothesis, \ref{it:R1}--\ref{it:R2} hold, and \(g\) is non-constant with bounded gradient, giving Assumption~\ref{ass:rigidity}\ref{it:U1}.
It remains to discharge the existence clause of
Assumption~\ref{ass:backward}: since \(\rho\), \(c\), and \(B\) are
bounded, the generator is globally Lipschitz in \(\left( y,z \right)\)
and of linear growth, and \(g\in C_b^1\left( \R^d \right)\), so a
Markovian solution exists by the classical results recalled in
Remark~\ref{rem:CYCZ}. Therefore, Theorem~\ref{thm:main} applies and
yields the density conclusion.
\end{proof}

\begin{rem}
Two comments are in order. First, the parameter \(\beta>0\) is what places the model within the framework of Section~\ref{sec:framework}. It turns the nonsmooth map \(p\mapsto\abs{p}\), for which the differentiability required of the generator fails, into a smooth, globally Lipschitz one. Second, the conclusion states that the marginal law of the activation variable \(Y_t\) is absolutely continuous at each fixed positive time. In particular, it has no atoms. When the uncertainty in predicted activation times is described probabilistically -- through probabilities of early or late activation, quantiles, or densities of the response rather than a single value -- this is the basic regularity that such summaries presuppose.
\end{rem}

\bibliographystyle{plain}
\bibliography{biblio}

\end{document}